\pdfoutput=1

\documentclass{amsart}
\usepackage{amsmath,amssymb,amsthm}
\usepackage{graphics,graphicx}
\usepackage{enumerate}
\usepackage{fancyhdr}
\usepackage[T1]{fontenc}
\usepackage{rotating}
\usepackage{MnSymbol}
\usepackage{tikz-cd}
\usepackage{tikz}
\usepackage[labelsep=none]{caption}
\usepackage{float}
\usepackage{chngcntr}

\counterwithout{equation}{section} 
\counterwithin{equation}{section}

\newcounter{dummy} \numberwithin{dummy}{section}

\newtheorem{theorem}[dummy]{Theorem}
\newtheorem*{theorem*}{Theorem}
\newtheorem{conjecture}[dummy]{Conjecture}
\newtheorem{lemma}[dummy]{Lemma}
\newtheorem{cor}[dummy]{Corollary}
\newtheorem{prop}[dummy]{Proposition}

\theoremstyle{remark}
\newtheorem{remark}[dummy]{Remark}

\newcommand{\fdavis}{\Phi}
\newcommand{\fchamber}{K^f}
\newcommand{\Q}{\mathbf{q}}
\DeclareMathOperator{\vcd}{vcd}

\date{}

\newcommand{\vertex}{\node[vertex]}
\title[Fattened Davis complex and weighted $L^2$-(co)homology]{The fattened Davis complex and weighted $L^2$-(co)homology of Coxeter groups}
\author{Wiktor J. Mogilski}
\address{Department of Mathematical Sciences, University of Wisconsin-Milwaukee, Milwaukee, WI 53211, USA}
\email{mogilski@uwm.edu}

\begin{document}

\begin{abstract}
Associated to a Coxeter system $(W,S)$ there is a contractible simplicial complex $\Sigma$ called the Davis complex on which $W$ acts properly and cocompactly by reflections. Given a positive real multiparameter $\Q$, one can define the weighted $L^2$-(co)homology groups of $\Sigma$ and associate to them a nonnegative real number called the weighted $L^2$-Betti number. Not much is known about the behavior of these groups when $\Q$ lies outside a certain restricted range, and weighted $L^2$-Betti numbers have proven difficult to compute. In this article we propose a program to compute the weighted $L^2$-(co)homology of $\Sigma$ by considering a thickened version of this complex. The program proves especially successful provided that the weighted $L^2$-(co)homology of certain infinite special subgroups of $W$ vanishes in low dimensions. We then use our complex to perform computations for many examples of Coxeter groups, in most cases providing explicit formulas for the weighted $L^2$-Betti numbers.
\end{abstract}
\maketitle
\section{Introduction}
Given a Coxeter system $(W,S)$ with nerve $L$, Davis defines a contractible simplicial complex $\Sigma_L$ on which $W$ acts properly and cocompactly. We provide the definition in Section $2$, but more details can be found in \cite{Davis,Davis1}. Given an $S$-tuple $\Q=(q_s)_{s\in S}$ of positive real numbers, where $q_s=q_{s'}$ if $s$ and $s'$ are conjugate in $W$, one defines the weighted $L^2$-(co)chain complex $L^2_\Q C_\ast(\Sigma_L)$ and the weighted $L^2$-(co)homology spaces $L^2_\Q H_k(\Sigma_L)$ (see \cite{DDJO}). They are special in the sense that they admit a notion of dimension: one can attach a nonnegative real number to each of the Hilbert spaces $L^2_\Q H_k(\Sigma_L)$ called the von Neumann dimension. Hence one can define weighted $L^2$-Betti numbers. We present a brief introduction to this theory in this article, but more details can be found in \cite{Davis,DDJO,Dymara}.

In \cite{DDJO}, weighted $L^2$-(co)homology was explicitly computed for CW-complexes on which a Coxeter group acts properly and cocompactly by reflections whenever $\Q\in\bar{\mathcal{R}}\cup\bar{\mathcal{R}}^{-1}$, where $\mathcal{R}$ denotes the region of convergence of the growth series of the Coxeter group. These formulas generalize those of Dymara \cite{Dymara} for $\Sigma_L$ and also compute the ordinary $L^2$-(co)homology of buildings of type $(W,S)$ with large integer thickness vectors. Furthermore, the Weighted Singer Conjecture (see Conjecture \ref{conj:Singer}) is proved in dimension less than or equal to $4$ when the Coxeter group is assumed to be right-angled, and more generally, it is proved that the odd dimensional Weighted Singer Conjecture implies the even dimensional Weighted Singer Conjecture. Weighted $L^2$-Betti numbers have proved difficult to compute in general, very little being known when $\Q\notin\bar{\mathcal{R}}\cup\bar{\mathcal{R}}^{-1}$. The aim of this article is to propose a method to compute them.

The results of this article can be summarized as follows.

\begin{itemize}
  \item We construct complex, called the fattened Davis complex, that serves as a successful program for computing the weighted $L^2$-(co)homology of many Coxeter groups.
  \item We observe that it is possible to use any acyclic complex on which the Coxeter group acts properly and cocompactly by reflections to compute the weighted $L^2$-Betti numbers of Coxeter groups. In particular, we can use a $\vcd W$-dimensional complex of Bestvina \cite{Bestvina} for computations.
  \item Using the above tools, we compute the weighted $L^2$-(co)homology for many classes Coxeter groups, in most cases providing explicit formulas for the weighted $L^2$-Betti numbers. Of note is that mostly all of the computations are performed for $\Q\geq\mathbf{1}$, and hence they compute the ordinary $L^2$-(co)homology of buildings associated to these Coxeter groups with thickness vector $\Q$.
  \item We prove a version of the Weighted Singer Conjecture in dimensions three and four for the case where $\Sigma_L$ is a manifold with (nonempty) boundary. In dimension four, we also prove a special case of the conjecture where the Coxeter group is $2$-spherical.
\end{itemize}

The article is structured as follows. We first introduce some basic notions and definitions, and then we proceed to construct what we call the fattened Davis complex. The idea is to ``fatten'' $\Sigma_L$ to a (homology) manifold with boundary so that we have standard algebraic topology tools (such as Poincar\'{e} duality) at our disposal. We carefully perform this fattening in such a way so that we can understand the weighted $L^2$-(co)homology of the boundary. In fact, understanding the weighted $L^2$-(co)homology of the boundary will simply amount to understanding the weighted $L^2$-(co)homology of certain non-spherical special subgroups of $W$. A large portion of the article will be dedicated to studying the structure and algebraic topology of the fattened Davis complex.

We then perform computations for many examples of Coxeter groups. For the purpose of stating our theorems, we label the edges of the nerve as follows, resulting in what we call the \emph{labeled nerve}. The vertices of the nerve are the generators for the Coxeter system, and on each edge we put the corresponding Coxeter relation between the generators on that edge. This allows us to recover the Coxeter system (up to isomorphism) by reading the presentation from the one-skeleton of the nerve.

We first restrict our attention to the case where the nerve $L$ of the Coxeter group is a graph. Suppose that the labeled nerve $L$ is the one-skeleton of an $n$-dimensional cell complex $\Lambda$, where $n\geq 2$. We say that a cell of labeled cell complex $\Lambda$ is \emph{Euclidean} if the corresponding special subgroup generated by the vertices of that cell is a Euclidean reflection group. With this terminology, we state the first main theorem.

\begin{theorem*}[Theorem \ref{thm:l2cellulationghs}]
Suppose that the labeled nerve $L$ is the one-skeleton of a cell complex that is a generalized homology $n$-sphere, $n\geq 2$, where all $2$-cells are Euclidean. If $\Q\geq\mathbf{1}$ then $L_\Q^2b_\ast(\Sigma_L)$ is concentrated in degree $2$.
\end{theorem*}

We also provide an explicit formula for $L_\Q^2b_2(\Sigma_L)$. For the special case when $n=2$, we explicitly compute the $L^2_\Q$-Betti numbers for \emph{all} $\Q$, not just for $\Q\geq\mathbf{1}$. We then derive various corollaries from this theorem, noting that if we place some restrictions on either the labels or the cell complex, then the formulas for the $L_\Q^2$-Betti numbers become relatively simple. We then discuss how Theorem \ref{thm:l2cellulationghs} can be used to produce Coxeter groups which satisfy the Weighted Singer Conjecture in dimensions three and four (see Theorem \ref{thm:singerghs}).

We then turn our attention to a class of Coxeter groups which in literature are sometimes called quasi-L\'{a}nner groups.

\begin{theorem*}[Theorem \ref{thm:l2QL}]
Suppose that $W$ acts properly but not cocompactly on hyperbolic space $\mathbb{H}^n$ by reflections with fundamental chamber an $n$-simplex of finite volume. Then $L_\Q^2b_k(\Sigma_L)=0$ whenever $k\geq n-1$ and $\Q\leq\mathbf{1}$, or $k\leq 1$ and $\Q\geq\mathbf{1}$.
\end{theorem*}

For $n=3$, we can use this theorem to explicitly compute formulas for the $L^2_\Q$-Betti numbers for all $\Q$: they are always concentrated in a single dimension.

Recall that a Coxeter system is \emph{$2$-spherical} if the one-skeleton of the corresponding nerve is a complete graph. In other words, this is equivalent to saying that for any distinct $s,t\in S$ we have the Coxeter relation $(st)^{m_{st}}=1$, where $m_{st}\geq 2$ is a finite natural number. A Coxeter group is \emph{Euclidean} if it acts properly and cocompactly by reflections on a Euclidean space of some dimension. We then perform computations for $2$-spherical Coxeter groups whose corresponding nerve is not necessarily a graph.

\begin{theorem*}[Theorem \ref{thm:genl2kn}]
Suppose that $(W,S)$ is infinite $2$-spherical with $|S|\geq 5$. Suppose furthermore that:
\begin{enumerate}
  \item For every $T\subseteq S$ with $|T|\geq 5$, $\vcd W_T\leq |T|-2$.
  \item Every infinite subgroup $W_T$, with $|T|=3,4$, is Euclidean or quasi-L\'{a}nner.
\end{enumerate}
If $\Q\geq\mathbf{1}$ then $L_\Q^2b_k(\Sigma_{L})=0 \text{ for } k<2$.
\end{theorem*}

We then discuss how the above theorem implies a specialized version of the Weighted Singer Conjecture for $2$-spherical Coxeter groups. We note that the computations of the above theorems not only rely on the fattened Davis complex, but also on Lemma \ref{lemma:pushingupcycles}. This lemma allows us to ``push'' known computations for $\Q=\mathbf{1}$ to $\Q\leq\mathbf{1}$, provided that the virtual cohomological dimension of the Coxeter group is lower than the dimension of the Davis complex. In fact, with the help of the work of Okun--Schreve \cite{OS}, we obtain the following theorem.

\begin{theorem*}[Theorem \ref{thm:singerfordisks}]
Suppose that the nerve $L$ is an $(n-1)$-disk. Then $$L_\Q^2H_k(\Sigma_L)=0 \text{ for } k\geq n-1 \text{ and } \Q\leq \mathbf{1}.$$
\end{theorem*}

Note that, when $n=3,4$, this theorem proves a version of the Weighted Singer Conjecture for the case where $\Sigma_L$ is a manifold with boundary.

\section*{Acknowledgements}
The author would like to thank his Ph.D. advisor Boris Okun for his masterful guidance and insightful discussions.

\section{Preliminaries}

\subsection{Coxeter systems and Coxeter groups.} A \emph{Coxeter matrix} $M=(m_{st})$ on a set $S$ is an $S\times S$ symmetric matrix with entries in $\mathbb{N}\cup\{\infty\}$ such that
$$m_{st}=\begin{cases} 1 &\mbox{if } s=t \\
\geq 2 & \mbox{otherwise.} \end{cases}$$
One can associate to $M$ a presentation for a group $W$ as follows. Let $S$ be the set of generators and let $\mathcal{I}=\{(s,t)\in S\times S | m_{st}\neq\infty\}.$ The set of relations for $W$ is $$R=\{(st)^{m_{st}}\}_{(s,t)\in\mathcal{I}}.$$
The group defined by the presentation $\left<S,R\right>$ is a \emph{Coxeter group} and the pair $(W,S)$ is a \emph{Coxeter system}. If all off-diagonal entries of $M$ are either $2$ or $\infty$, then $W$ is \emph{right-angled}.

Given a subset $T\subset S$, define $W_T$ to be the subgroup of $W$ generated by the elements of $T$. Then $(W_T,T)$ is a Coxeter system. Subgroups of this form are \emph{special subgroups}. $W_T$ is a \emph{spherical subgroup} if $W_T$ is finite and, in this case, $T$ is a \emph{spherical subset}. If $W_T$ is infinite, then $T$ is \emph{non-spherical}. We will let $\mathcal{S}$ denote the poset of spherical subsets (the partial order being inclusion).

Given $w\in W$, call an expression $w=s_1s_2\dotsm s_n$ \emph{reduced} if there exists no integer $k<n$ with $w=s_1's_2'\dotsm s_k'.$ We define the \emph{length} of $w$, denoted by $l(w)$, to be the integer $n$ so that $w=s_1s_2\dotsm s_n$ is a reduced expression for $w$. Given a subset $T\subset S$ and an element $w\in W$, the special coset $wW_T$ contains a unique element of shortest length. This element is said to be \emph{$(\emptyset, T)$-reduced}.

\subsection{Growth series.} Suppose that $(W,S)$ is a Coxeter system. Let $\mathbf{t}:=(t_s)_{s\in S}$ denote an $S$-tuple of indeterminates, where $t_s=t_{s'}$ if $s$ and $s'$ are conjugate in $W$. If $s_1s_2\dotsm s_n$ is a reduced expression for $w$, define $t_w$ to be the monomial $$t_w:=t_{s_1}t_{s_2}\dotsm t_{s_n}.$$

Note that $t_w$ is independent of choice of reduced expression due to Tits' solution to the word problem for Coxeter groups (see the discussion at the beginning of \cite[Chapter 17]{Davis}). The \emph{growth series} of $W$ is the power series in $\mathbf{t}$ defined by $$W(\mathbf{t})=\sum_{w\in W} t_w.$$

The \emph{region of convergence} $\mathcal{R}$ for $W(\mathbf{t})$ is defined to be $$\mathcal{R}:=\{\mathbf{t}\in (0,+\infty)^S \mid W(\mathbf{t}) \text{ converges}\}.$$

For each $T\subset S$, we denote the growth series of the special subgroup $W_T$ by $W_T(\mathbf{t})$, and define $\mathbf{t}^{-1}:=(t_s^{-1})_{s\in S}$. We record the following formula for later computations.

\begin{theorem}[{\cite[Theorem 17.1.9]{Davis}}]
\label{thm:growthseriesformula}
$$\frac{1}{W(\mathbf{t})}={\Large \sum_{T\in\mathcal{S}}}\frac{(-1)^{|T|}}{W_T(\mathbf{t}^{-1})}.$$
\end{theorem}

Note that if $W$ is finite, then $W(\mathbf{t})$ is a polynomial with integral coefficients. Thus an immediate consequence of the above formula is that $W(\mathbf{t})$ is a rational function in $\mathbf{t}$.

\subsection{Homology manifolds.} A space $X$ is a \emph{homology $n$-manifold} if it has the same local homology groups as $\mathbb{R}^n$, i.e. that for each $x\in X$
$$H_k(X,X- x)=\begin{cases} \mathbb{Z} &\mbox{if } k=n \\
 0 & \mbox{otherwise.} \end{cases}$$
The pair $(X,\partial X)$ with $\partial X$ closed in $X$ is a \emph{homology $n$-manifold with boundary} if it has the same local homology groups as does a manifold with boundary, i.e., the following conditions hold:
\begin{itemize}
  \item $X-\partial X$ is a homology $n$-manifold,
  \item $\partial X$ is a homology $(n-1)$-manifold,
  \item for each $x\in\partial X$, the local homology groups $H_\ast(X,X- x)$ all vanish.
\end{itemize}

$X$ is a \emph{generalized homology $n$-sphere}, abbreviated $GHS^n$, if it is a homology $n$-manifold with the same homology as $S^n$. Similarly, the pair $(X,\partial X)$ is a \emph{generalized homology $n$-disk}, abbreviated $GHD^n$, if it is a homology $n$-manifold with boundary with the same homology as the pair $(D^n,S^{n-1})$. Note that the cone on a generalized homology sphere is a generalized homology disk.

\subsection{Mirrored spaces.} A \emph{mirror structure} over a set $S$ on a space $X$ is a family of subspaces $(X_s)_{s\in S}$ indexed by $S$. Then $X$ is a \emph{mirrored space over $S$}. Put $X_\emptyset=X$, and for each nonempty subset $T\subseteq S$, define the following subspaces of $X$: $$X_T:=\bigcap_{s\in T} X_s,\hskip2mm X^T:=\bigcup_{s\in T} X_s.$$

If $(W,S)$ is a Coxeter system and $X$ is a mirrored space over $S$, then the mirror structure $(X_s)_{s\in S}$ is \emph{$W$-finite} if $X_T=\emptyset$ for all non-spherical $T\subseteq S$.

\subsection{Mirrored homology manifolds with corners.} Suppose that $X$ is a mirrored space over a finite set $S$. $X$ is an \emph{$S$-mirrored homology $n$-manifold with corners} if every nonempty $X_T$ is a homology $(n-|T|)$-manifold with boundary $\partial X_T=\bigcup_{U\supsetneq T} X_U$. By taking $T=\emptyset$, this definition implies that the pair $(X,\partial X)$ is a homology $n$-manifold with boundary.

Given a Coxeter system $(W,S)$, we set $S'=S\cup \{e\}$, where $e$ is the identity element of $W$. We now say that $T\subseteq S'$ is spherical if and only if $T-\{e\}$ is spherical. A mirrored space $X$ over the set $S'$ with $W$-finite mirror structure $(X_s)_{s\in S'}$ is a \emph{partially $S$-mirrored homology $n$-manifold with corners} if every nonempty $X_T$ is a homology $(n-|T|)$-manifold with boundary $\partial X_T=\bigcup_{U\supsetneq T} X_U$. To summarize, we simply have defined the non-$S$-mirrored part of $X$ to be an auxiliary mirror corresponding to the identity element of $W$.

\subsection{Basic construction.} Suppose that $(W,S)$ is a Coxeter system and that $X$ is a mirrored space over $S$. Set $W_\emptyset=W$ and as before, for each nonempty subset $T\subset S$, let $W_T$ be the subgroup of $W$ generated by $T\subset S$. Put $S(x):=\{s\in S\mid x\in X_s\}$. Define an equivalence relation $\sim$ on $W\times X$ by $(w,x)\sim (w',y)$ if and only if $x=y$ and $w^{-1}w'\in W_{S(x)}$. Give $W\times X$ the product topology and let $\mathcal{U}(W,X)$ denote the quotient space: $$\mathcal{U}(W,X)=(W\times X)/ \sim.$$

$\mathcal{U}(W,X)$ is the \emph{basic construction} and $X$ is the \emph{fundamental chamber}. There is a natural $W$-action on $W\times X$, and this action respects the equivalence relation, hence the $W$-action on $W\times X$ descends to a $W$-action on $\mathcal{U}(W,X)$.

We will be interested in conditions on $X$ which guarantee that the basic construction produces a homology $n$-manifold with boundary. But first, we consider the following proposition, as the proof is similar to the main result of this subsection.

\begin{prop}[Compare {\cite[Proposition 10.7.5]{Davis}}]
\label{prop:basicconstrmnfld}
Suppose that $(W,S)$ is a Coxeter system and that $X$ is an $S$-mirrored homology $n$-manifold with corners with $W$-finite mirror structure. Then $\mathcal{U}(W,X)$ is a homology $n$-manifold.
\end{prop}
\begin{proof}
Without loss of generality suppose that $x\in X$. By excision, we need to show that the local homology groups $H_\ast (U,U-x)$ are correct for some neighborhood $U$ of $x$ in $\mathcal{U}(W,X)$. If $x\in X-\partial X$ then we are done since $X-\partial X$ is a homology $n$-manifold and $x$ does not lie in any mirror. As before, set $S(x)=\{s\in S\mid x\in X_s\}$ and suppose that $|S(x)|\geq 1$.

Let $V$ be a neighborhood of $x$ in $X$. For each $s\in S(x)$, set $V_s=V\cap X_s$, and give $V$ the mirror structure $\{V_s\}_{s\in S(x)}$. Note that, for each $T\subseteq S(x)$, $V_T=V\cap X_T$, where as before, $X_T=\bigcap_{s\in T} X_s$. Now, $x\in X_{S(x)}$, so for each $T\subset S(x)$, $x\in\partial X_T$ ($X_T$ is by assumption a homology $(n-|T|)$-manifold with boundary and $X_{S(x)}\subseteq\partial X_T$). Furthermore, $x$ does not lie in $\partial X_{S(x)}$. Therefore by excision, it follows that for each $T\subset S(x)$, the local homology groups $H_\ast(V_T,V_T-x)$ vanish, and $H_\ast(V_{S(x)},V_{S(x)}-x)$ is concentrated in dimension $n-|S(x)|$ and $\mathbb{Z}$ in that dimension.

Now, define

\begin{align*}
Z &:=  V\cup \textnormal{Cone}(V-x) \\
Z_s &:=  V_s\cup \textnormal{Cone}(V_s-x)
\end{align*}

So, $Z$ has the mirror structure $\{Z_s\}_{s\in S(x)}$. Since $V$ is a neighborhood of $x$ in $X$, and $x\in\partial X$, it follows that the local homology groups $H_\ast(V,V-x)$ vanish. In particular, $H_\ast(V)\cong H_\ast(V-x)$, and the Mayer--Vietoris sequence, along with the five lemma, implies that $Z$ is acyclic. Similarly, for each $T\subset S(x)$, since the local homology groups $H_\ast(V_T,V_T-x)$ vanish, it follows that $Z_T$ is acyclic. Since $H_\ast(V_{S(x)},V_{S(x)}-x)$ is concentrated in dimension $n-|S(x)|$ and $\mathbb{Z}$ in that dimension, that Mayer--Vietoris sequence again implies that the same is true for $H_\ast(Z_{S(x)})$. In particular, $Z_{S(x)}$ has the same homology as $S^{n-|S(x)|}$.

We now finish the proof by applying the following lemma:
\begin{lemma}[{\cite[Corollary 8.2.5]{Davis}}]
$\mathcal{U}(W_{S(x)},Z)$ has the same homology as $S^n$ if and only if there is a unique spherical subset $R\subseteq S(x)$ satisfying the following three conditions:
\begin{description}
  \item[(a)] $W_{S(x)}$ decomposes as $W_{S(x)}=W_R\times W_{S(x)-R}$.
  \item[(b)] For all spherical $T'\subseteq S(x)$  with $T'\neq R$, $(Z,Z^{T'})$ is acyclic.
  \item[(c)] $(Z,Z^R)$ has the same homology as $(D^n,S^{n-1})$.
\end{description}
\end{lemma}

We apply the lemma to $R=S(x)$. Condition (a) is then satisfied vacuously, so we wish to show (b) and (c). For $T\subseteq R$, consider the cover of $Z^T$ by the mirrors $\{Z_s\}_{s\in T}$. Note that for each $U\subset R$, the intersection of mirrors $Z_U$ is acyclic. The nerve of this cover is a simplex on $U$, and in particular is contractible. The Acyclic Covering Lemma \cite[Theorem 4.4, Ch VII]{Brown} then implies that $Z^U$ is acyclic. Note that $Z_R$ has the same homology as $S^{n-|R|}$, so a similar spectral sequence argument also implies that $Z^{R}$ has the same homology as $S^{n-1}$.

Now, set $U=\mathcal{U}(W_{R},V)$. Since $\mathcal{U}(W_{R},Z)=U\cup \textnormal{Cone}(U-x)$ and $\mathcal{U}(W_{R},Z)$ has the same homology as $S^n$, it follows that $H_\ast(U,U-x)$ is concentrated in dimension $n$ and $\mathbb{Z}$ in that dimension. Therefore $U$ is our desired neighborhood.
\end{proof}

\begin{prop}
\label{prop:basicconstrmnfldbdry}
Suppose that $(W,S)$ is a Coxeter system and suppose that $X$ is a partially $S$-mirrored homology $n$-manifold with corners. Set $Y=X_e$ and give $Y$ the induced mirror structure $(Y_s)_{s\in S}$, where $Y_s:=Y\cap X_s$. Then $\mathcal{U}(W,X)$ is a homology $n$-manifold with boundary $\partial\mathcal{U}(W,X)=\mathcal{U}(W,Y).$
\end{prop}
\begin{proof}
Set $\mathcal{U}=\mathcal{U}(W,X)$ and $\partial\mathcal{U}=\mathcal{U}(W,Y)$. Proposition \ref{prop:basicconstrmnfld} guarantees that $\partial\mathcal{U}$ is a homology $(n-1)$-manifold. This is because $Y=X_e$, and $X_e$ (with its induced $S$-mirror structure) is an $S$-mirrored homology $(n-1)$-manifold with corners. Similarly, Proposition \ref{prop:basicconstrmnfld} implies that $\mathcal{U}-\partial\mathcal{U}$ is a homology $n$-manifold, since $\mathcal{U}-\partial\mathcal{U}=U(W,Z)$, where $Z=X-Y$ (with its induced $S$-mirror structure) is an $S$-mirrored homology $n$-manifold with corners. It remains to show that for each $x\in\partial\mathcal{U}$, the local homology groups $H_\ast(\mathcal{U},\mathcal{U}-x)$ vanish.

 Suppose that $x\in\partial\mathcal{U}$. Without loss of generality, we can assume that $x\in Y\subset\partial X$. If $x$ does not lie in any mirror $(X_s)_{s\in S}$, then we are done by excision. So, suppose $|S(x)|\geq 1$ (recall $S(x)=\{s\in S\mid x\in X_s\}$) and let $V$ be a neighborhood of $x$ in $X$. We now give $V$ the $S$-mirror structure as in the proof of Proposition \ref{prop:basicconstrmnfld}, noting that the only difference between that proof and the current situation is the fact that the local homology groups $H_\ast(V_{S(x)},V_{S(x)}-x)$ vanish. This is because, since $x\in Y$ and $|S(x)|\geq 1$, it follows that $x\in \partial X_{S(x)}$. Now, following the proof of Proposition \ref{prop:basicconstrmnfld} line by line, the only difference now is that $Z_{S(x)}$ is acyclic (as opposed to having the homology of $S^{n-1}$ as before). This then implies that $\mathcal{U}(W_{S(x)},Z)$ is acyclic \cite[Corollary 8.2.8]{Davis}, which in turn implies that the local homology groups $H_\ast(\mathcal{U},\mathcal{U}-x)$ vanish.
\end{proof}

\subsection{Cell complexes.} A \emph{cell} is the convex hull of finitely many points in $\mathbb{R}^n$. A \emph{cell complex} is a collection of cells $\Lambda$ where

\begin{enumerate}[(i)]
  \item if $C\in\Lambda$ and $F$ is a face of $C$, then $F\in\Lambda$,
  \item for any two cells $C_1,C_2\in \Lambda$, either $C_1\cap C_2=\emptyset$ or $C_1\cap C_2$ is a common face of $C_1$ and $C_2$,
  \item $\Lambda$ is locally finite, i.e. each cell in $\Lambda$ is contained in only finitely many other cells in $\Lambda$.
\end{enumerate}

A \emph{cellulation} of a space $X$ is a homeomorphism $f$ from a cell complex $\Lambda$ onto $X$. We will subdue the homeomorphism $f$ and just say that $\Lambda$ is a cellulation of $X$.

\subsection{The $(\Lambda,S)$-chamber.} Suppose that $\Lambda$ is a cell complex with vertex set $S$ and let $\mathcal{F}(\Lambda)$ denote the poset of cells of $\Lambda$, including the empty set. Let $P:=|\mathcal{F}(\Lambda)|$ denote the geometric realization of the poset $\mathcal{F}(\Lambda)$. For each $T\in\mathcal{F}(\Lambda)$, define $P_T:=|\mathcal{F}(\Lambda)_{\geq T}|$ and $\partial P_T:=|\mathcal{F}(\Lambda)_{>T}|$, so each $P_T$ is the cone on $b\textnormal{Link}(T,\Lambda)$, the barycentric subdvision of $\textnormal{Link}(T,\Lambda)$. In particular, taking $T=\emptyset$, we have that $P$ is the cone on $b\Lambda$, with cone point corresponding to $\emptyset$. For each $s\in S$, put $P_s:=P_{\{s\}}$. This endows $P$ with the mirror structure $(P_s)_{s\in S}$. $P$ is the \emph{$(\Lambda,S)$-chamber}.

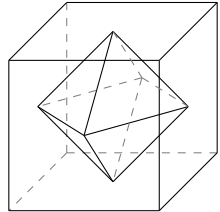
\begin{figure}[H]
\[\begin{tikzpicture}[scale=1]
\tikzstyle{vertex}=[circle, fill=black, draw, inner sep=0pt, minimum size=1pt]

\coordinate (c0) at (0,0,0);
\coordinate (c1) at (0,2,0);
\coordinate (c2) at (0,2,2);
\coordinate (c3) at (0,0,2);
\coordinate (c4) at (2,0,0);
\coordinate (c5) at (2,2,0);
\coordinate (c6) at (2,2,2);
\coordinate (c7) at (2,0,2);

\coordinate (o1) at (1,2,1);  
\coordinate (o2) at (1,1,2);  
\coordinate (o3) at (0,1,1);  
\coordinate (o4) at (1,1,0);  
\coordinate (o5) at (1,0,1);  
\coordinate (o6) at (2,1,1);  

\path[gray,dashed]
    (c0) edge (c1)
    (c0) edge (c3)
    (c0) edge (c4)
    ;
\path
    (c1) edge (c2)
    (c1) edge (c5)
    (c2) edge (c3)
    (c2) edge (c6)
    (c3) edge (c7)
    (c4) edge (c5)
    (c4) edge (c7)
    (c5) edge (c6)
    (c6) edge (c7)
    ;

\path
    (o1) edge (o2)
    (o1) edge (o3)
    (o1) edge (o6)
    (o2) edge (o5)
    (o2) edge (o6)
    (o2) edge (o3)
    (o3) edge (o5)
    (o5) edge (o6)
    ;
\path[gray,dashed]
    (o1) edge (o4)
    (o4) edge (o6)
    (o4) edge (o3)
    (o4) edge (o5)
;

\end{tikzpicture}\]
\caption{:\hskip2mm$(\Lambda,S)$-chamber when $\Lambda$ is the boundary complex of an octahedron}
\end{figure}

Note that if $\Lambda$ is a $GHS^{n-1}$, then the link of every cell $\sigma$ in $\Lambda$ is a $GHS^{n-\dim\sigma-2}$. It follows that $P$ is a $GHD^n$ and that for each $T\in\mathcal{F}(\Lambda)$, the pair $(P_T,\partial P_T)$ is a $GHD^{n-\dim\sigma_T-1}$, where $\sigma_T$ is the geometric cell in $\Lambda$ spanned by $T$.

\subsubsection{Neighborhoods of faces.}\label{subsection:nghbdsoffaces} Let $\sigma_T$ denote the geometric cell spanned by the vertex set $T$ in $\Lambda$, and let $b\sigma_T$ denote its barycentric subdivision. By definition, $b\sigma_T$ is the $(\partial\sigma_T,T)$-chamber, and in particular, $\sigma_T$ has a natural mirror structure over $T$.

$P$ is itself a flag simplicial complex, and for each $T\in\mathcal{F}(\Lambda)$, $P_T$ is a subcomplex of $P$. Hence $P_T-\bigcup_{U\supset T} P_U$ has a neighborhood of the form $\sigma_T\ast P_T$, the join of $\sigma_T$ and $P_T$. Following the join lines for a little while, it follows that $P_T-\bigcup_{U\supset T} P_U$ has neighborhoods of the form $\textnormal{Cone}(\sigma_T)\times P_T$. We record this fact, as we will use it in an upcoming construction.

\subsection{The Davis complex.} Suppose that $(W,S)$ is a Coxeter system and, as before, denote by $\mathcal{S}$ the poset of all spherical subsets of $S$, partially ordered by inclusion. $\mathcal{S}$ is an abstract simplicial complex with vertex set $S$. Let $L$ be the geometric realization of the abstract simplicial complex $\mathcal{S}$ and $K$ be the $(L,S)$-chamber. In this special situation, $K$ is called the \emph{Davis chamber} and $L$ is called the \emph{nerve} of $(W,S)$.

For each $s\in S$ define $$K_s:=|\mathcal{S}_{\geq\{s\}}|.$$ So, $K_s$ is the union of simplices in $K$ with minimum vertex $\{s\}$. The family $(K_s)_{s\in S}$ is a mirror structure on $K$.

The \emph{Davis complex} $\Sigma_L$ associated to the nerve $L$ is now defined to be $\Sigma_L:=\mathcal{U}(W,K)$.

Note that $\Sigma_L$ is naturally a simplicial complex, the simplicial structure of $K$ inducing a simplicial structure on $\Sigma_L$, and moreover, it is proved in \cite{Davis1} that $\Sigma_L$ is contractible. Furthermore, if $L$ is a triangulation of an $(n-1)$-sphere, then $\Sigma_L$ is an $n$-manifold.

\subsubsection{The labeled nerve.} There is a natural way to label the edges of $L$ so that the Coxeter system $(W,S)$ can be recovered (up to isomorphism) from $L$. Let $E(L)$ denote the set of edges of $L$. We define the labeling map $m:E(L)\rightarrow \{2,3,...\}$ by sending $\{s,t\}\rightarrow m_{st}$, where $m_{st}\in\mathbb{N}$ and $(st)^{m_{st}}=1$. $L$ with this labeling map is the \emph{labeled nerve}.

\subsubsection{Right-angled cones.} Let $c$ denote a point and $L$ be the labeled nerve. Consider the join $L'=c\ast L$, where all of the new edges are labeled by $2$. $L'$ is called the \emph{right-angled cone on $L$}. Note that the corresponding Coxeter system to $L'$ is $(W\times\mathbb{Z}_2, S\cup\{c\})$, and $\Sigma_{L'}=\Sigma_L\times [-1,1].$

\subsubsection{The Coxeter cellulation.} \label{section:coxcell} The Davis complex also admits a decomposition into \emph{Coxeter cells}. For each $T\in\mathcal{S}$, let $v_T$ denote the corresponding barycenter in $K$. Let $c_T$ denote the union of simplices $c\subset\Sigma_L$ such that $c\cap K_T=v_T$. The boundary of $c_T$ is then cellulated by $wc_U$, where $w\in W_T$ and $U\subset T$. With its simplicial structure, the boundary $\partial c_T$ is the Coxeter complex corresponding to the Coxeter system $(W_T,T)$, which is a sphere since $W_T$ is finite. It follows that $c_T$ and its translates are disks, which are called \emph{Coxeter cells of type $T$}. We denote $\Sigma_L$ with this decomposition into Coxeter cells by $\Sigma_{cc}$. Note that $\Sigma_{cc}$ is a regular CW complex with with poset of cells that can be identified with $W\mathcal{S}:=\{wW_U\mid w\in W, T\in\mathcal{S}\}$. The simplicial structure on $\Sigma_L$ is the geometric realization of the poset $W\mathcal{S}$, hence $\Sigma_L$ is the barycentric subdivision of $\Sigma_{cc}$.

\section{The fattened Davis complex}

We will now construct a complex which is a ``fattened'' version of the Davis complex. This new thickened complex will be a homology manifold with boundary possessing the Davis complex as a $W$-equivariant retract. For the remainder of this article, we suppose that $W$ is an infinite Coxeter group.

\subsection{Construction.} Given a Coxeter system $(W,S)$, we find a compact $P$ with mirror structure $(P_s)_{s\in S}$ as follows. Let $P^\ast$ be a cell complex with vertex set $S$ that is a $GHS^{n-1}$, with $n-1>\dim(L)$, such that the nerve is a subcomplex of $P^\ast$. Take $P$ to be the $(P^\ast,S)$-chamber.

Denote by $\mathcal{P}$ the collection of proper nonempty subsets $T$ of $S$ with $P_T\neq\emptyset$. We denote by $\mathcal{N}_P$ the subcollection of $\mathcal{P}$ corresponding to non-spherical subsets. For $T\in\mathcal{P}$, we denote a neighborhood of the face $P_T$ by $N(P_T)$ and the corresponding closed neighborhood by $\bar{N}(P_T)$.

We begin by building a regular neighborhood of $\partial P$ in $P$. Start by choosing neighborhoods of codimension-$n$ faces so that for any two codimension-$n$ faces $P_U$ and $P_V$ we have $\bar{N}(P_U)\cap \bar{N}(P_V)=\emptyset$. Then we choose neighborhoods of codimension-$(n-1)$ faces so that for any two codimension-$(n-1)$ faces $P_U$ and $P_V$ we have:

\begin{equation}\label{eq:neighborhoods} \bar{N}(P_U)\cap \bar{N}(P_V)\subset N(P_U\cap P_V).\end{equation}

If $U\cup V\notin\mathcal{P}$, then we take $N(P_U\cap P_V)=\emptyset.$ We proceed inductively, employing condition (\ref{eq:neighborhoods}) at each step until we obtain the collection $\{N(P_T)\}_{T\in\mathcal{P}}$. This collection gives us a regular neighborhood of $\partial P$.

Finally, we realize the neighborhoods $\{N(P_T)\}_{T\in\mathcal{P}}$ in the above construction as $\{N_T\times P_T \}_{T\in\mathcal{P}}$, where $N_T$ is a neighborhood of the cone point in $\textnormal{Cone}(\sigma_T)$ and $\sigma_T$ is the geometric cell in $P^\ast$ spanned by $T$ (note that we can always do this, see the discussion in the previous section).

We now define $$\fchamber:=P-\bigcup_{\substack{T\in\mathcal{N}_P}}N(P_T).$$ We call $\fchamber$ the \emph{fattened Davis chamber}.

Note that the mirror structure $(P_s)_{s\in S}$ on $P$ induces a mirror structure $(\fchamber_s)_{s\in S}$ on $\fchamber$. Define $\fdavis_L:=\mathcal{U}(W,\fchamber)$. We call $\fdavis_L$ the \emph{fattened Davis complex}.

Given a $T\in\mathcal{N}_P$, we denote by $\fchamber(T)$ the fattened Davis chamber corresponding to $\sigma_T$ and Coxeter system $(W_T,T)$ (recall that the geometric cell $\sigma_T$ has a natural $W_T$ mirror structure).

\begin{remark}
\label{remark:simplexpolytope}
For any Coxeter system $(W,S)$, one can always find a $P^\ast$ for the above construction: simply let $P^\ast$ be the boundary of the standard $(|S|-1)$-dimensional simplex $\Delta^{|S|-1}$. Then $P$ is the barycentric subdivision of $\Delta^{|S|-1}$, and the Davis chamber $K$ can then be viewed as a subcomplex of the barycentric subdivision of $P$ spanned by the barycenters of spherical faces. One can see this using the language of posets. Note that $K$ is the geometric realization of the poset $\mathcal{S}$ and $P$ is the geometric realization of the poset of proper subsets of $S$. The natural inclusion of posets now induces the desired inclusion of $K$ into $P$. The mirror structure $(K_s)_{s\in S}$ on $K$ is now induced by the mirror structure $(P_s)_{s\in S}$ on $P$. In this case $\mathcal{U}(W,P)$ is the traditional Coxeter complex, and we are essentially viewing $\Sigma_L$ as a subcomplex of the barycentric subdivision of the Coxeter complex.
\end{remark}

\subsection{Properties of $\fdavis_L$.} $W$ is assumed to be infinite, so via the choice of $P$ for construction, the Davis chamber is the subcomplex of $P$ spanned by vertices of $P$ corresponding to spherical faces. Hence we have the following inclusions: $K\subset\fchamber\subset P$ (See Figure \ref{figure:davisinfdavis}).

\begin{figure}[H]
\[\begin{tikzpicture}[scale=1.5]
\tikzstyle{vertex}=[circle, fill=black, draw, inner sep=0pt, minimum size=1pt]
\tikzstyle{every node}=[draw,circle,fill=white,minimum size=0pt,
                            inner sep=0pt]
\begin{scope}[rotate=-23]

\path[transparent]
    (1,1,1) edge node[pos=0.2] (r1) {} node[pos=0.8] (s1) {} (1,-1,-1)
    (1,1,1) edge node[pos=0.2] (r3) {} node[pos=0.8] (u1) {} (-1,1,-1)
    (1,-1,-1) edge node[pos=0.2] (s2) {} node[pos=0.8] (t2) {} (-1,-1,1)
    (-1,-1,1) edge node[pos=0.2] (t3) {} node[pos=0.8] (u3) {} (-1,1,-1)
    (1,-1,-1) edge node[pos=0.2] (s3) {} node[pos=0.8] (u2) {} (-1,1,-1)
    (1,1,1) edge node[pos=0.2] (r2) {} node[pos=0.8] (t1) {} (-1,-1,1)
    ;

\path[transparent]
    (r3) edge (r1)
    (s1) edge (s2)
    (s2) edge node[pos=0.6] (s11) {} (s3)
    (s3) edge node[pos=0.4] (s22) {} (s1)
    (t2) edge (t3)
    (u1) edge node[pos=0.6] (u11) {} (u2)
    (u2) edge node[pos=0.4] (u22) {} (u3)
    (u3) edge (u1)
    (r1) edge node[pos=0.6] (r11) {} (r2)
    (r2) edge node[pos=0.4] (r22) {} (r3)
    (t1) edge node[pos=0.4] (t22) {} (t2)
    (t3) edge node[pos=0.6] (t11) {} (t1)
    ;

\path[gray,dashed]
    (r11) edge (t22)
    (r22) edge (t11)
    (r11) edge (r22)
    (t11) edge (t22)
    ;
\path[gray,dashed]
    (t3) edge (t11)
    (t2) edge (t22)
    (r3) edge (r22)
    (r1) edge (r11)
    ;
\path[black!60!white, dotted]
    (1,1,1) edge (-1,-1,1)
    ;

\vertex (b) at (0,0,0) {};

\vertex (f1) at (1/3,-1/3,1/3) {};
\vertex (f2) at (-1/3,1/3,1/3) {};
\vertex (f3) at (-1/3,-1/3,-1/3) {};
\vertex (f4) at (1/3,1/3,-1/3) {};

\vertex (e1) at (1,0,0) {};
\vertex (e2) at (0,-1,0) {};
\vertex (e3) at (0,1,0) {};
\vertex (e4) at (-1,0,0) {};

\draw[fill=gray!70] (e1.center) -- (f1.center) -- (b.center) -- cycle;
\draw[fill=gray!70] (e1.center) -- (f4.center) -- (b.center) -- cycle;
\draw[fill=gray!70] (f4.center) -- (e3.center) -- (b.center) -- cycle;
\draw[fill=gray!70] (e3.center) -- (f2.center) -- (b.center) -- cycle;
\draw[fill=gray!70] (f2.center) -- (e4.center) -- (b.center) -- cycle;
\draw[fill=gray!70] (e4.center) -- (f3.center) -- (b.center) -- cycle;
\draw[fill=gray!70] (f3.center) -- (e2.center) -- (b.center) -- cycle;
\draw[fill=gray!70] (e2.center) -- (f1.center) -- (b.center) -- cycle;

\path
    (b) edge (f1)
    (b) edge (f2)
    (b) edge (f3)
    (b) edge (f4)
    (b) edge (e1)
    (b) edge (e2)
    (b) edge (e3)
    (b) edge (e4)
    (e1) edge (f1)
    (e1) edge (f4)
    (e2) edge (f1)
    (e2) edge (f3)
    (e4) edge (f3)
    (e4) edge (f2)
    (e3) edge (f2)
    (e3) edge (f4)
    ;

\path
    (s2) edge (s11)
    (s2) edge (s1)
    (s1) edge (s22)
    (t3) edge (t2)
    (u3) edge (u1)
    (u3) edge (u22)
    (u1) edge (u11)
    (r1) edge (r3)
    ;

\path
    (s11) edge (s22)
    (s11) edge (u22)
    (s22) edge (u11)
    (u11) edge (u22)
    ;

\path
    (t2) edge (s2)
    (s1) edge (r1)
    (r3) edge (u1)
    (t3) edge (u3)
    ;

\path[dotted]
    (1,-1,-1) edge (s1)
    (1,-1,-1) edge (s2)
    (1,-1,-1) edge (-1,1,-1)
    (-1,1,-1) edge (u1)
    (-1,1,-1) edge (u3)
    (1,1,1) edge (r3)
    (1,1,1) edge (r1)
    (-1,-1,1) edge (t2)
    (-1,-1,1) edge (t3)
    ;

\end{scope}
\end{tikzpicture}\]
\caption{:\text{ }$K\subset \fchamber\subset P$ when $W=D_\infty\times D_\infty$ and $P=\Delta^3$}
\label{figure:davisinfdavis}
\end{figure}
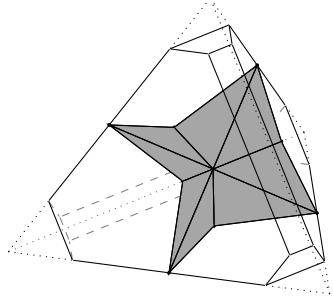

Note that there is a face preserving deformation retraction of $\fchamber$ onto $K$, thus we have the following:

\begin{prop}
\label{prop:defretract}
$\fdavis_L$ $W$-equivariantly deformation retracts onto $\Sigma_L$.
\end{prop}

\begin{prop}
\label{prop:fdavismnfld}
$\fdavis_L$ is a locally compact contractible homology $n$-manifold with boundary $\partial\fdavis_L$.
\end{prop}
\begin{proof}
Since $\Sigma_L$ is contractible, it follows from Proposition \ref{prop:defretract} that $\fdavis_L$ is contractible. Moreover, $\fchamber$ is compact since it is closed in $P$ ($P$ is compact), so $\fdavis_L$ is locally compact.

As before, give $\fchamber$ the mirror structure $(\fchamber_s)_{s\in S}$ induced from $P$, and declare $\fchamber_e=\partial\fchamber-\bigcup_{T\in\mathcal{S}_{>\emptyset}} (\fchamber_T-\partial\fchamber_T)$, where $e$ is the identity element of $W$. According to Proposition \ref{prop:basicconstrmnfldbdry}, it remains to show that $\fchamber$ is a partially $S$-mirrored homology manifold with corners. Let $S'=S\cup\{e\}$ and note that by construction $\fchamber_T=\emptyset$ if and only if $T$ is not spherical. So, we are done if we show that for every spherical $T\subset S'$, $\fchamber_T$ has dimension $n-|T|$.

If $e\notin T$, then we are done since $(P_T,\partial P_T)$ is a $GHD^{n-|T|}$. This is because $P$ is by definition the $(P^\ast,S)$-chamber and the nerve $L$ was assumed to be a subcomplex of $P^\ast$. Hence, since $T$ is spherical, $\sigma_T$, the geometric cell in $P^\ast$ corresponding to $T$, is a simplex of dimension $|T|-1$. Therefore the dimension of $P_T$ is equal to $n-\dim \sigma_T-1=n-|T|$.

If $e\in T$, then $U=T-\{e\}$ is spherical, and by the above discussion $\fchamber_U$ has dimension $n-|U|=n-|T|+1$. Then $\fchamber_T=\fchamber_U\cap\fchamber_e=\partial\fchamber_U$ has dimension $n-|T|$.
\end{proof}

\begin{remark} If $P=\Delta^{|S|-1}$, then the Coxeter complex $\mathcal{U}(W,P)$ is a PL-manifold away from faces with infinite stabilizers. This is because the links of faces corresponding to spherical subsets $T$ are homeomorphic to the Coxeter complex of the corresponding group $W_T$. Since $W_T$ is finite, this Coxeter complex is homeomorphic to a sphere of appropriate dimension. Since we obtain $\fdavis_L$ by removing neighborhoods of non-spherical faces (faces with infinite stabilizers), it follows that $\fdavis_L$ is a PL-manifold with boundary.
\end{remark}

\subsection{The structure of $\partial\fdavis_L$.} The main goal of this section is to understand the structure of $\partial\fdavis_L$. The first proposition will tell us that $\partial\fchamber$ can be broken up into pieces, each of which has a nice product structure. This decomposition of $\partial\fchamber$ then leads us to a cover of $\partial\fdavis_L$ which will be used to study the algebraic topology of $\partial\fdavis_L$.

For $T\in\mathcal{N}_P$, define $$C_T=\partial N(P_T)-\bigcup_{\substack{U\in\mathcal{N}_P}} N(P_U),$$ $$\Lambda_T=P_T-\bigcup_{\substack{U\in\mathcal{N}_P\\T\subset U}}N(P_U).$$

\begin{prop}
\label{prop:boundarycomponentsareproducts}
\begin{enumerate}[(i)]
\item Suppose that $U,V\in\mathcal{N}_P$. Then $C_U\cap C_V\neq\emptyset$ if and only if $U\subset V$ or $V\subset U$.
  \item If $T\in\mathcal{N}_P$ then $$C_T\approx \fchamber(T)\times\Lambda_T.$$
  \item Suppose that $T_1,T_2\in\mathcal{N}_P$ with $T_1\subset T_2$. Then $$C_{T_1}\cap C_{T_2}\approx\fchamber(T_1)\times\Lambda_{T_2}.$$
\end{enumerate}
\end{prop}
\begin{proof}
For (i), one implication is obvious. If $U\subset V$, then $P_V$ is a face of $P_U$. Thus $C_U\cap C_V\neq\emptyset$. For the reverse implication, suppose that $U\not\subset V$ and $V\not\subset U$. By construction and condition (\ref{eq:neighborhoods}), either $\bar{N}(P_U)\cap \bar{N}(P_V)=\emptyset$ or $\bar{N}(P_U)\cap \bar{N}(P_V)\subset N(P_U\cap P_V)$. The former case immediately implies that $C_U\cap C_V=\emptyset$, and the latter case implies that the intersection $\partial N(P_U)\cap \partial N(P_V)$ is removed at some point in the construction of the fattened Davis chamber, hence $C_U\cap C_V=\emptyset$.

For (ii), recall that we have realized the collection $\{N(P_T)\}_{T\in\mathcal{N}_P}$ as neighborhoods $\{N_T\times P_T \}_{T\in\mathcal{N}_P}$, where $N_T$ is a neighborhood of the cone point in $\textnormal{Cone}(\sigma_T)$.

Now, for each $U\subset T$, let $\alpha_U$ denote the face in $\sigma_T$ corresponding to $P_U$. More precisely, $\sigma_T$ has a $W_T$ mirror structure, and $\alpha_U$ is the intersection of mirrors corresponding to $U\subset T$. We can express the neighborhoods in the construction of $\fchamber(T)$ as neighborhoods $\{\alpha_U\times N'_U\}_{\substack{U\in\mathcal{N}_P\\U\subset T}}$, where $N'_U$ is a neighborhood of the cone point in $\textnormal{Cone}(\textnormal{Lk}(\alpha_U,\sigma_T))$. Here $\textnormal{Lk}(\alpha_U,\sigma_T)$ denotes the link of the face $\alpha_U$ in $\sigma_T$. In particular,

$$\fchamber(T)=\sigma_T-\bigcup_{\substack{U\in\mathcal{N}_P\\U\subset T}}\alpha_U\times N'_U.$$

Now, we have that $\textnormal{Lk}(\alpha_U,\sigma_T)\approx \sigma_U$, so $N'_U\approx N_U$. Hence

$$\fchamber(T)\approx \sigma_T-\bigcup_{\substack{U\in\mathcal{N}_P\\U\subset T}} P_U\times N_U.$$

Moreover, we can write $\Lambda_T$ and $C_T$ as
$$\Lambda_T=P_T-\bigcup_{\substack{U\in\mathcal{N}_P\\T\subset U}}P_U\times N_U,$$

$$C_T=(\sigma_T\times P_T)-\bigcup_{\substack{U\in\mathcal{N}_P\\U\neq T}} P_U\times N_U.$$

We now show that $C_T\approx \fchamber(T)\times\Lambda_T$. Note that $\fchamber(T)\times\Lambda_T=(\fchamber(T)\times P_T)\cap (\sigma_T\times\Lambda_T)$, so we begin unwinding definitions. We first observe that

$$\fchamber(T)\times P_T\approx \left(\sigma_T-\bigcup_{\substack{U\in\mathcal{N}_P\\U\subset T}} P_U\times N_U\right)\times P_T\approx(\sigma_T\times P_T)-\bigcup_{\substack{U\in\mathcal{N}_P\\U\subset T}} P_U\times N_U.$$

This is because $P_T$ is a face of each of the $P_U$'s. Similarly, we have

$$\sigma_T\times\Lambda_T= \sigma_T\times \left(P_T-\bigcup_{\substack{U\in\mathcal{N}_P\\T\subset U}}P_U\times N_U\right)\approx (\sigma_T\times P_T)-\bigcup_{\substack{U\in\mathcal{N}_P\\T\subset U}} P_U\times N_U.$$

This follows from the fact that $P_U$'s are faces of $P_T$. Thus we have shown that $\fchamber(T)\times\Lambda_T=(\fchamber(T)\times P_T)\cap (\sigma_T\times\Lambda_T)\approx C_T$, therefore proving (ii).

We now prove (iii). By (ii), $$C_{T_1}\cap C_{T_2}\approx(\fchamber(T_1)\cap\fchamber(T_2))\times(\Lambda_{T_1}\cap\Lambda_{T_2}).$$ It now simply remains to unwind the definitions. Since $T_1\subset T_2$, it follows that $P_{T_2}$ is a face of $P_{T_1}$. In particular, $\sigma_{T_1}\cap \sigma_{T_2}=\sigma_{T_1}$ and hence
\begin{align*}
\fchamber(T_1)\cap\fchamber(T_2)&\approx \sigma_{T_1}\cap \sigma_{T_2}-\bigcup_{\substack{U,V\in\mathcal{N}_P\\U\subset T_1\\V\subset T_2}} N(P_U)\cup N(P_V)\\
&\approx \sigma_{T_1}-\bigcup_{\substack{U\in\mathcal{N}_P\\U\subset T_1}} N(P_U)\\
&\approx \fchamber(T_1)
\end{align*}

A similar computation shows that $\Lambda_{T_1}\cap\Lambda_{T_2}\approx \Lambda_{T_2}$, thus completing the proof of the proposition.
\end{proof}

\begin{prop}
\label{prop:boundaryfdavis}$$\partial\fdavis_L=\bigcup_j \bigsqcup_{T\in\mathcal{N}_P^{(j)}}\mathcal{U}(W,C_T),$$ where $\mathcal{N}_P^{(j)}=\{T\in \mathcal{N}_P\mid\text{Card}(T)=j\}$.
\end{prop}
\begin{proof}
The fact that one can decompose $\partial\fdavis_L$ in this way is clear by construction, and the second union is in fact a disjoint union by Proposition \ref{prop:boundarycomponentsareproducts} (i).
\end{proof}

\section{Weighted $L^2$-(co)homology}

In this section we present a brief introduction to weighted $L^2$-(co)homology. Further details can be found in \cite{Davis,DDJO,Dymara}. We then compile some results pertaining to the weighted $L^2$-(co)homology of the Davis complex $\Sigma_L$, as well as for $\fdavis_L$ and $\partial\fdavis_L$.

Let $(W,S)$ be a Coxeter system. For the remainder of this article, let $\Q=(q_s)_{s\in S}$ denote an $S$-tuple of positive real numbers satisfying $q_s=q_{s'}$ whenever $s$ and $s'$ are conjugate in $W$. Set $\Q^{-1}=(q^{-1}_s)_{s\in S}$. If $w=s_1\dotsm s_n$ is a reduced expression for $w\in W$, we define $q_w:=q_{s_1}\dotsm q_{s_n}.$

\subsection{Hecke--von Neumann algebras.} Let $\mathbb{R}(W)$ denote the group algebra of $W$, and let $\{e_w\}_{w\in W}$ denote the standard basis on $\mathbb{R}(W)$ (here $e_w$ denotes the characteristic function of $\{w\}$). Given a multiparameter $\Q$ of positive real numbers as above, we deform the standard inner product on $\mathbb{R}(W)$ to an inner product
$$\left<e_w,e_{w'}\right>_\Q=\begin{cases} q_w &\mbox{if } w=w' \\
0 & \mbox{otherwise.} \end{cases}$$

Using the multiparameter $\Q$, one can give $\mathbb{R}(W)$ the structure of a $\emph{Hecke algebra}$. We will denote $\mathbb{R}(W)$ with this inner product and Hecke algebra structure by $\mathbb{R}_\Q (W)$, and $L^2_\Q(W)$ will denote the Hilbert space completion of $\mathbb{R}_\Q (W)$ with respect to $\left<\hskip1mm,\hskip1mm\right>_\Q$. There is a natural anti-involution on $\mathbb{R}_\Q (W)$, which implies that there is an associated \emph{Hecke-von Neumann algebra} $\mathcal{N}_\Q(W)$ acting on the right on $L^2_\Q(W)$. It is the algebra of all bounded linear endomorphisms of $L^2_\Q(W)$ which commute with the left $\mathbb{R}_\Q (W)$-action.

Define the $\emph{von Neumann trace}$ of $\phi\in\mathcal{N}_\Q(W)$ by $\textnormal{tr}_{\mathcal{N}_\Q}(\phi):=\left<e_1\phi,e_1\right>_\Q$, and similarly for an $(n\times n)$-matrix with coefficients in $\phi\in\mathcal{N}_\Q(W)$ by taking the sum of the von Neumann traces of elements on the diagonal. This allows us to attribute an nonnegative real number called the \emph{von Neumann dimension} for any closed subspace of an $n$-fold direct sum of copies of $L^2_\Q(W)$ which is stable under the $\mathbb{R}_\Q (W)$-action, called a \emph{Hilbert $\mathcal{N}_\Q$-module}. If $V\subseteq (L^2_\Q(W))^n$ is a Hilbert $\mathcal{N}_\Q$-module, and $p_V: (L^2_\Q(W))^n\rightarrow (L^2_\Q(W))^n$ is the orthogonal projection onto $V$ (note that $p_V\in\mathcal{N}_\Q(W)$), then define $$\dim_{\mathcal{N}_\Q} V:=\textnormal{tr}_{\mathcal{N}_\Q}(p_V).$$

\subsection{Weighted $L^2$-(co)homology.} Suppose $(W,S)$ is a Coxeter system and that $X$ is a mirrored finite $CW$-complex over $S$. Set $\mathcal{U}=\mathcal{U}(W,X)$. We first orient the cells of $X$ and equivariantly extend this orientation to $\mathcal{U}$ in such a way so that if $\sigma$ is a positively oriented cell of $X$, then $w\sigma$ is positively oriented for each $w\in W$.

We define a measure on the $w$-orbit of an $i$-cell $\sigma\in X$ by $$\mu_\Q(w\sigma)=q_u,$$ where $u$ is $(\emptyset, S(\sigma))$-reduced and $S(\sigma):=\{s\in S|\sigma\subseteq X_s\}.$ This extends to a measure on the $i$-cells $\mathcal{U}^{(i)}$, which we also denote by $\mu_\Q$. Define the \emph{$\Q$-weighted $i$-dimensional $L^2$-(co)chains} on $\mathcal{U}$ to be the Hilbert space: $$L_\Q^2C_i(\mathcal{U})=L_\Q^2C^i(\mathcal{U})=L^2(\mathcal{U}^{(i)},\mu_\Q).$$ These are infinite $W$-equivariant square summable (with respect to $\mu_\Q$) real-valued $i$-chains. The inner product is given by $$\left<f,g\right>_\Q=\sum_\sigma f(\sigma)g(\sigma)\mu_\Q(\sigma),$$ and we denote the induced norm by $||\text{ }||_\Q$.

The boundary map $\partial_i: L_\Q^2C_i(\mathcal{U})\rightarrow L_\Q^2C_{i-1}(\mathcal{U})$ and coboundary map $\delta^i: L_\Q^2C_i(\mathcal{U})\rightarrow L_\Q^2C_{i+1}(\mathcal{U})$ are defined by the usual formulas, however there is one caveat: they are not adjoints with respect to this inner product whenever $\Q\neq\mathbf{1}$. Thus one remedies this issue by perturbing the boundary map $\partial_i$ to $\partial_i^\Q$: $$\partial_i^\Q(f)(\sigma^{i-1})=\sum_{\sigma^{i-1}\subset\alpha^{i}} [\sigma:\alpha]\mu_\Q(\alpha)\mu_\Q^{-1}(\sigma) f(\alpha).$$

A simple computation shows that $\partial_i^\Q$ is the adjoint of $\delta^i$ with respect to the weighted inner product, hence $\left(L_\Q^2C_\ast(\mathcal{U}),\partial_i^\Q\right)$ is a chain complex. We now define the \emph{reduced $\Q$-weighted $L^2$-(co)homology} by $$L_\Q^2H_i(\mathcal{U})=\text{Ker}\partial_i^\Q/\overline{\text{Im}\partial_{i+1}^\Q},$$ $$L_\Q^2H^i(\mathcal{U})=\text{Ker}\delta^i/\overline{\text{Im}\delta^{i-1}}.$$

The Hodge Decomposition implies that $L_\Q^2H^i(\mathcal{U})\cong L_\Q^2H_i(\mathcal{U})=\ker\partial^{\Q}_i \cap \ker\delta^i$ and versions of Eilenberg-Steenrod axioms hold for this homology theory. There is also a weighted version of Poincar\'{e} duality: If $\mathcal{U}$ is a locally compact homology $n$-manifold with boundary $\partial\mathcal{U}$, then $$L_\Q^2H_i(\mathcal{U})\cong L_{\Q^{-1}}^2H_{n-i}(\mathcal{U},\partial\mathcal{U}).$$

One can also assign the von Neumann dimension to each of the Hilbert spaces $L^2_\Q H_i(\mathcal{U})$ (as they are Hilbert $\mathcal{N}_\Q$-modules). We denote this by $L_\Q^2b_i(\mathcal{U})$ and call it the \emph{$i$-th $L_\Q^2$-Betti number of $\mathcal{U}$}. We then define the \emph{weighted Euler characteristic of $\mathcal{U}$}: $$\chi_\Q(\mathcal{U})=\sum (-1)^iL_\Q^2b_i(\mathcal{U}).$$

\subsection{An alternate definition of $L_\Q^2$-Betti numbers.} As discussed in \cite[$\S$6]{DO2}, there is an alternate approach in defining $L_\Q^2$-Betti numbers using the ideas of L\"{u}ck \cite{Luckbook}. The main point is that there is an equivalence of categories between the category of Hilbert $\mathcal{N}_\Q$-modules and projective modules of $\mathcal{N}_\Q$. Hence one can define $\dim_{\mathcal{N}_\Q} M$ for a finitely generated projective $\mathcal{N}_\Q$-module $M$ which agrees with the dimension of the corresponding Hilbert $\mathcal{N}_\Q$-module. So, $\dim_{\mathcal{N}_\Q} M$ for an arbitrary $\mathcal{N}_\Q$-module is then defined to be the dimension of its projective part.

As before, suppose $(W,S)$ is a Coxeter system and that $X$ is a mirrored finite $CW$-complex over $S$. Set $\mathcal{U}=\mathcal{U}(W,X)$. As in \cite{Luckbook}, define $H_\ast^W(\mathcal{U},\mathcal{N}_\Q(W))$ to be the homology of the $\mathcal{N}_\Q(W)$-chain complex $C_\ast^W(\mathcal{U},\mathcal{N}_\Q(W)):=\mathcal{N}_\Q(W)\otimes_{\mathbb{R}_\Q (W)} C_\ast(\mathcal{U})$, where $C_\ast(\mathcal{U})$ is the cellular chain complex of $\mathcal{U}$ with the induced $\mathbb{R}_\Q (W)$-structure. We then define $$L_\Q^2b_i(\mathcal{U}):=\dim_{\mathcal{N}_\Q}H_i^W(\mathcal{U},\mathcal{N}_\Q(W)).$$ This definition does in fact agree with the previous one, and the advantage of this definition is that we do no need to take closures of images as in the definition of reduced $\Q$-weighted $L^2$-(co)homology (this is particularly useful when dealing with spectral sequences).

\subsection{Some results for $\Sigma_L$.} In this section we begin by stating some previous results on the weighted $L^2$-(co)homology of $\Sigma_L$. We start with the following result of Dymara, which explicitly computes $L_\Q^2b_0(\Sigma_L)$.

\begin{prop}[{\cite[Theorem 7.1, Theorem 10.3]{Dymara}}]
\label{prop:betti0}
$L_\Q^2b_0(\Sigma_L)\neq 0$ if and only if $\Q\in\mathcal{R}$. Moreover, when $\Q\in\mathcal{R}$, $L_\Q^2b_k(\Sigma_L)=0$ for $k>0$.
\end{prop}

Dymara also computes the weighted Euler characteristic of $\Sigma_L$, revealing the connection between weighted $L^2$-(co)homology of $\Sigma_L$ and the growth series of the corresponding Coxeter group $W$.

\begin{prop}[{\cite[Corollary 3.4]{Dymara}}]
\label{prop:eulerchar}
$$\chi_\Q(\Sigma_L)=\frac{1}{W(\Q)}.$$
\end{prop}

Recall that $\Sigma_{cc}$ denotes $\Sigma_L$ with the Coxeter cellulation. The following proposition states that if we compute the weighted $L^2$-(co)homology with respect to either cellulation, then we get the same answer.

\begin{prop}[{\cite[Theorem 5.5]{Dymara}}]
$$L^2_\Q H_\ast(\Sigma_L)\cong L^2_\Q H_\ast(\Sigma_{cc}).$$
\end{prop}

In conjunction with Proposition \ref{prop:eulerchar}, the following theorem explicitly computes the weighted $L^2$-(co)homology of Coxeter groups which act properly and cocompactly by reflections on Euclidean space.

\begin{theorem}[{\cite[Corollary 14.5]{DDJO}}]
\label{thm:weightedeuclidean}
Suppose that $W$ is a Euclidean reflection group with nerve $L$.
\begin{itemize}
  \item If $\Q\leq\mathbf{1}$, then $L_\Q^2H_\ast(\Sigma_L)$ is concentrated in dimension $0$.
  \item If $\Q\geq\mathbf{1}$, then $L_\Q^2H_\ast(\Sigma_L)$ is concentrated in dimension $n$.
\end{itemize}
\end{theorem}

The following lemma says that we can compute the weighted $L^2$-Betti numbers of any acyclic complex of the form $\mathcal{U}(W,X)$, with $X$ finite, on which $W$ acts properly, and get the same answer. Thus we will sometimes write $L_\Q^2 b_k(W)$ instead of $L_\Q^2 b_k(\Sigma_L)$ to denote the \emph{$k$-th $L_\Q^2$-Betti number of $W$}.

\begin{lemma}
\label{lemma:l2homologyacycliccomplexes}
Let $(W,S)$ be a Coxeter system and suppose that $X$ and $X'$ are finite mirrored CW-complexes with $\mathcal{U}(W,X)$ and $\mathcal{U}(W,X')$ both acyclic and both admitting proper $W$-action. Then for every $k\geq 0$, $$L_\Q^2 b_k(\mathcal{U}(W,X))= L_\Q^2 b_k(\mathcal{U}(W,X')).$$
\end{lemma}
\begin{proof}
Set $\mathcal{U}=\mathcal{U}(W,X)$ and $\mathcal{U}'=\mathcal{U}(W,X')$. Since $\mathcal{U}$ and $\mathcal{U}'$ are both acyclic, it follows that the respective cellular chain complexes $C_\ast(\mathcal{U})$ and $C_\ast(\mathcal{U}')$ are are chain homotopic. This chain homotopy induces a chain homotopy of the chain complexes $C_\ast^W(\mathcal{U},\mathcal{N}_\Q(W))$ and $C_\ast^W(\mathcal{U'},\mathcal{N}_\Q(W))$.
\end{proof}

In fact, Bestvina constructed such a complex for any finitely generated Coxeter group.

\begin{theorem}[\cite{Bestvina}]
\label{thm:bestvinacpx} Let $W$ be a finitely generated Coxeter group. Then $W$ acts properly and cocompactly on an acyclic $\vcd W$-dimensional complex of the form $\mathcal{U}(W,X)$.
\end{theorem}

\begin{cor}
\label{cor:vanishingabovevcd}
Let $(W,S)$ be a Coxeter system. Then $$L_\Q^2b_k(W)=0 \text{ for } k>\vcd W.$$
\end{cor}
\begin{proof}
We can use the acyclic $\vcd W$-dimensional complex of Theorem \ref{thm:bestvinacpx} to compute the weighted $L^2$-Betti numbers of $W$. Lemma \ref{lemma:l2homologyacycliccomplexes} now completes the proof.
\end{proof}

We now prove a lemma which is crucial for later computations.

\begin{lemma}
\label{lemma:pushingupcycles}
Let $n=\vcd W$ and suppose and that $L_\mathbf{1}^2b_n(W)=0$. Then $$L_\Q^2b_k(W)=0 \text{ for } k\geq n \text{ and } \Q\leq\mathbf{1}.$$
\end{lemma}
\begin{proof}
By Corollary \ref{cor:vanishingabovevcd}, we obtain vanishing for $k>n$. Now, suppose for a contradiction that $L_\Q^2b_n(W)\neq 0$ for $\Q<\mathbf{1}$. Let $B_W$ denote the complex of Theorem \ref{thm:bestvinacpx}. Lemma \ref{lemma:l2homologyacycliccomplexes} says that we can compute weighted $L^2$-Betti numbers of $W$ with respect to the complex $B_W$. In particular, $L_\Q^2b_n(W)=L_\Q^2b_n(B_W)$ and we can choose a nontrivial element $\psi\in L_\Q^2H_n(B_W)$. Thus $\psi$ is a cycle under the weighted boundary map $\partial^\Q$. Consider the isomorphism of Hilbert spaces $$m_\Q: L_\Q^2C_n(B_W)\rightarrow L_{\Q^{-1}}^2C_n(B_W)$$ defined by $m_\Q(f(\sigma))=\mu_\Q(\sigma)f(\sigma)$. In particular, $m_\Q\psi\in L_{\Q^{-1}}^2C_n(B_W)$ and since $\Q^{-1}>\mathbf{1}$, $$||m_\Q\psi||_{\mathbf{1}}\leq ||m_\Q\psi||_{\Q^{-1}}<\infty.$$
 Hence $m_\Q\psi\in L_\mathbf{1}^2C_n(B_W)$.
 \begin{figure}[H]
\[\begin{tikzpicture}
\tikzstyle{vertex}=[circle, draw, inner sep=0pt, minimum size=0pt]
\tikzstyle{every node}=[circle,inner sep=0pt,minimum size=0pt]

\vertex (x1) at (0,0) {};
\vertex (x2) at (10,0) {};

\path[|->]
    (x1) edge node[pos=0,below=6pt] (0) {$\mathbf{0}$} node[pos=.2,below=5pt] (q1) {$\mathbf{q}^{\text{ }}$} node[pos=.5,below=6pt] (1) {$\mathbf{1}$} node[pos=.8,below=5pt] (q2) {$\mathbf{q}^{-1}$} (x2)
    ;

\path[->]
    (q1) edge [out=-50,in=210] node[pos=.5,above=2pt] (map) {$m_\mathbf{q}$} (q2)
;

\draw (q1) edge node[pos=1,above=5.5pt] (psi) {$\psi$} ++(0,0.45);
\draw (q2) edge node[pos=1,above=0.5pt] (mpsi) {$m_{\mathbf{q}}\psi$} ++(0,0.6);
\draw (1) edge node[pos=1,above=0.5pt] (mpsi1) {$m_{\mathbf{q}}\psi$} ++(0,0.45);

\path[->]
    (mpsi) edge (mpsi1);
;
\end{tikzpicture}\]
\caption{:\hskip2mm Schematic for the proof of Lemma \ref{lemma:pushingupcycles}}
\end{figure}
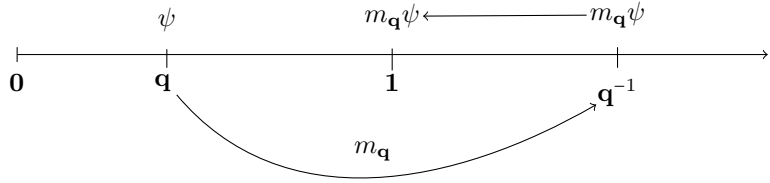
 Now, a simple computation shows that $\partial=m_\Q\partial^\Q m_\Q^{-1}$ and since $\psi$ is a cycle under $\partial^\Q$, $m_\Q\psi$ is a cycle under $\partial$, the standard $L^2$-boundary operator. Moreover, since $B_W$ is $n$-dimensional, $m_\Q\psi$ is trivially a cocycle. Thus we have produced a nontrivial element of $L_{\mathbf{1}}^2H_n(B_W)$, a contradiction.
\end{proof}

\subsection{Algebraic topology of $\fdavis_L$ and $\partial\fdavis_L$}
We now turn our attention to studying the algebraic topology of $\fdavis_L$ and $\partial\fdavis_L$. We first begin with a corollary of Proposition \ref{prop:defretract}.

\begin{cor}
\label{cor:Hsigma=Hfdavis}
$$L_\Q^2H_\ast(\Phi_L)\cong L_\Q^2H_\ast(\Sigma_L).$$
\end{cor}

Not only does $\fdavis_L$ have the same weighted $L^2$-(co)homology as $\Sigma_L$, but by Proposition \ref{prop:fdavismnfld}, $\fdavis_L$ is a locally compact homology manifold with boundary. Thus we have weighted Poincar\'{e} duality for $\fdavis_L$ at our disposal. With this in mind, we prove the following lemma.

\begin{lemma}
\label{lemma:h1bdry}
Suppose that $(W,S)$ is a Coxeter system with $\vcd W=m$ and that $\fdavis_L$ is a homology $n$-manifold with boundary with $L_\Q^2b_1(\partial\fdavis_L)=0$.
\begin{enumerate}[(i)]
\item If $n-m=1$ and $L_{\Q^{-1}}^2b_m(\fdavis_L)=0$ then $L_\Q^2b_1(\Sigma_L)=0$.
\item If $n-m\geq 2$ then $L_\Q^2b_1(\Sigma_L)=0$.
\end{enumerate}
\end{lemma}
\begin{proof}
Consider the long exact sequence for the pair $(\fdavis_L,\partial\fdavis_L)$:
$$\begin{tikzcd}[column sep=small]
\cdot\cdot\cdot \arrow{r} & L_\Q^2H_1(\partial\fdavis_L) \arrow{r} & L_\Q^2H_1(\fdavis_L) \arrow{r} & L_\Q^2H_1(\fdavis_L,\partial\fdavis_L) \arrow{r} & \cdot\cdot\cdot
\end{tikzcd}$$
By weighted Poincar\'{e} duality $$L_\Q^2H_1(\fdavis_L,\partial\fdavis_L)\cong L_{\Q^{-1}}^2H_{n-1}(\fdavis_L).$$

Now, by assumption $L_\Q^2H_1(\partial\fdavis_L)=0$, so by weak exactness we must show that $L_{\Q^{-1}}^2H_{n-1}(\fdavis_L)=0$. We will then be done by Corollary \ref{cor:Hsigma=Hfdavis}, which says that $L_\Q^2H_1(\Sigma_L)=L_\Q^2H_1(\fdavis_L)=0$.

For (i), we have that $L_{\Q^{-1}}^2b_m(\fdavis_L)=0$. Since $n-m=1$, we have that $m=n-1$, so it follows that $L_{\Q^{-1}}^2H_{n-1}(\fdavis_L)=0$.
For (ii), we have that $n-m\geq 2$, so $n-1\geq m+1$. Since $\vcd W=m$, Corollary \ref{cor:vanishingabovevcd} implies that $$L_{\Q^{-1}}^2H_{n-1}(\Sigma_L)=L_{\Q^{-1}}^2H_{n-1}(\fdavis_L)=0.$$
\end{proof}

We devote the remainder of the section to studying the algebraic topology of $\partial\fdavis_L$. The following is a corollary of Proposition \ref{prop:boundarycomponentsareproducts}.

\begin{cor}
\label{cor:l2boundarycomponents}
\begin{enumerate}[(i)]
\item If $T\in\mathcal{N}_P$, then for every $k\geq 0$ $$L_\Q^2b_k(\mathcal{U}(W,C_T))=L_\Q^2b_k(\fdavis_{L_T})=L_\Q^2b_k(\Sigma_{L_T}),$$ where $L_T$ is the subcomplex of $L$ corresponding to the subgroup $W_T$.
\item Suppose that $T_1,T_2\in\mathcal{N}_P$ with $T_1\subset T_2$. Then for every $k\geq 0$ $$L_\Q^2b_k(\mathcal{U}(W,C_{T_1})\cap\mathcal{U}(W,C_{T_2}))=L_\Q^2b_k(\fdavis_{L_{T_1}})=L_\Q^2b_k(\Sigma_{L_{T_1}}),$$ where $L_{T_1}$ is the subcomplex of $L$ corresponding to the subgroup $W_{T_1}$.
\end{enumerate}
\begin{remark}
The $L_\Q^2$-Betti numbers on the center and the right of the equations in $(i)$ and $(ii)$ are computed with respect to the special subgroups $W_T$ (respectively $W_{T_1}$) of $W$, while the ones on the far left side of the equations are computed with respect to $W$.
\end{remark}
\end{cor}
\begin{proof}
We prove only (i) as the proof of (ii) is similar. Proposition \ref{prop:boundarycomponentsareproducts} implies that $C_T\approx \fchamber(T)\times\Lambda_T$ as mirrored spaces, where $\Lambda_T$ is contractible and has no mirror structure. Therefore $\mathcal{U}(W,C_T)$ is $W$-equivariantly homotopy equivalent to $\mathcal{U}(W,\fchamber(T))$. Now, $L_\Q^2H_\ast\left(\mathcal{U}(W,\fchamber(T))\right)$ is just the completion of $$L_\Q^2(W)\otimes_{\mathbb{R}_\Q(W_T)} L_\Q^2H_\ast\left(\mathcal{U}(W_T,\fchamber(T))\right),$$ so for every $k\geq 0$, $$L_\Q^2b_k\left(\mathcal{U}(W,\fchamber(T))\right)=L_\Q^2b_k\left(\mathcal{U}(W_T,\fchamber(T))\right)=L_\Q^2b_k(\fdavis_{L_T}).$$
\end{proof}

Consider the cover $\mathcal{V}=\{\mathcal{U}(W,C_T)\}_{T\in\mathcal{N}_P}$ of $\partial\fdavis_L$ in Proposition \ref{prop:boundaryfdavis}. The cover $\mathcal{V}$ will have intersections of variable depth, so we obtain a spectral sequence following \cite[Ch. VII, \S 3,4]{Brown}:

\begin{prop}
\label{prop:specialspectralsequence}
There is a Mayer--Vietoris type spectral sequence converging to $H_\ast^W(\partial\fdavis_L,\mathcal{N}_\Q(W))$ with $E_1$-term: $$E_1^{i,j}=\bigoplus_{\substack{\sigma\in \textnormal{Flag}(\mathcal{N}_P) \\ \dim\sigma=i}}H_j^W(\mathcal{U}(W,C_{\min\sigma}),\mathcal{N}_\Q(W)).$$
\end{prop}
\begin{proof}
Let $N(\mathcal{V})$ denote the nerve of the cover $\mathcal{V}$. It is the abstract simplicial complex whose vertex set is $\mathcal{N}_P$ and whose simplices are the non-empty subsets $\sigma\subset \mathcal{N}_P$ such that the intersection $V_\sigma=\bigcap_{T\in\sigma}\mathcal{U}(W,C_T)$ is non-empty. Following \cite[Ch. VII, \S 3,4]{Brown}, there is a Mayer--Vietoris type spectral sequence converging to $H_\ast^W(\partial\fdavis_L,\mathcal{N}_\Q(W))$ with $E_1$-term:

$$E_1^{i,j}=\bigoplus_{\substack{\sigma\in N(\mathcal{V}) \\ \dim\sigma=i}}H_j^W(V_\sigma,\mathcal{N}_\Q(W)).$$

We have that $V_\sigma\neq\emptyset$ if and only if $\bigcap_{T\in\sigma}C_T\neq\emptyset$, and applying Proposition \ref{prop:boundarycomponentsareproducts} inductively, this happens if and only if the vertices of $\sigma$ form a chain $T_{i_1}\subset T_{i_2}\subset\dotsm\subset T_{i_k}$. This observation shows that $N(\mathcal{V})=\textnormal{Flag}(\mathcal{N}_P)$. Now, applying Proposition \ref{prop:boundarycomponentsareproducts} inductively, it follows that $V_\sigma\approx\mathcal{U}(W,C_{T_{i_1}})$. Hence $H_\ast^W(V_\sigma,\mathcal{N}_\Q(W))=H_\ast^W(\mathcal{U}(W,C_{T_{i_1}}),\mathcal{N}_\Q(W))$, so the terms in the spectral sequence are the ones claimed.
\end{proof}

For later computations, note that Corollary \ref{cor:l2boundarycomponents} implies:

\begin{align*}
L_\Q^2b_\ast(\mathcal{U}(W,C_{\min\sigma})) & =\dim_{\mathcal{N}_\Q}H_\ast^W(\mathcal{U}(W,C_{\min\sigma})) \\
 & =L_\Q^2b_\ast(\fdavis_{L_{\min\sigma}})\\
 &=L_\Q^2b_\ast(\Sigma_{L_{\min\sigma}}).
\end{align*}

\section{Computations}

In this section we will use the fattened Davis complex to make concrete computations. We first begin by considering the case where the nerve $L$ of the Coxeter system $(W,S)$ is a graph. Note that for this special case $\Sigma_L$ is two-dimensional. We then briefly discuss how we can use our computations to produce examples of Coxeter groups for which the Weighted Singer Conjecture holds. We then direct our attention to quasi-L\'{a}nner groups, and finish with computations for $2$-spherical Coxeter groups whose corresponding nerves are no longer restricted to be graphs.

Let $K_n$ denote the complete graph on $n$ vertices. Recall that a Coxeter system is \emph{$2$-spherical} if the one-skeleton of its nerve is $K_n$ for some $n$. For the purpose of figures and examples, we will distinguish the special case where the labeled nerve $L=K_n(3)$, where $K_n(3)$ denotes the complete graph on $n$ vertices with every edge labeled by $3$.

Unless stated otherwise, the standing assumption in this section is that $\Q\geq\mathbf{1}$.

\subsection{The case where $L$ is a graph.} Suppose that the labeled nerve $L$ is the one-skeleton of an $n$-dimensional cell complex $\Lambda$, where $n\geq 2$. We say that a $2$-cell of $\Lambda$ is \emph{Euclidean} if the corresponding special subgroup generated by the vertices of that cell is a Euclidean reflection group. Note that the only possible labels on a Euclidean cell are $m_{st}\in\{2,3,4,6\}$.

Before proving the main theorem of this section, we begin with a lemma. The special case of the lemma when $\Q=\mathbf{1}$ is closely related to a result of Schroeder \cite[Theorem 4.6]{Schroeder1}. We provide an argument which is analogous to that of Schroeder in his proof.

\begin{lemma}
\label{lemma:l2spherecellulations}
Suppose that the labeled nerve $L$ is the one-skeleton of a cellulation of $S^2$. Then
    $$L_\Q^2b_2(\Sigma_L)=0\text{ for } \Q\leq \mathbf{1}.$$
\end{lemma}
\begin{proof}
In light of Lemma \ref{lemma:pushingupcycles}, we must show that $L_\mathbf{1}^2b_2(\Sigma_L)=0$. We begin by building $L$ to a triangulation of $S^2$ by coning on empty $2$-cells and labeling the new edges by $2$'s, at each step keeping track of the $L_\mathbf{1}^2$-(co)homology with a Mayer--Vietoris sequence. More precisely, start with $T_1\subset S$ corresponding to an empty $2$-cell $L_{T_1}$ in $L$ and denote by $CL_{T_1}$ the right-angled cone on $L_{T_1}$. The corresponding special subgroup $W_{T_1}$ is infinite, and it acts properly and cocompactly by reflections on either $\mathbb{R}^2$ or $\mathbb{H}^2$. In both cases $L_\mathbf{1}^2H_2(\Sigma_{L_{T_1}})=0$ and hence the K\"{u}nneth formula implies that $L_\mathbf{1}^2H_2(\Sigma_{CL_{T_1}})=0$. We have the following Mayer--Vietoris sequence:

$$\cdot\cdot\cdot \longrightarrow L_\mathbf{1}^2H_2(\Sigma_{L_{T_1}}) \longrightarrow  L_\mathbf{1}^2H_2(\Sigma_{CL_{T_1}}) \oplus L_\mathbf{1}^2H_2(\Sigma_L)\overset{f_1}{\longrightarrow} L_\mathbf{1}^2H_2(\Sigma_{L\cup CL_{T_1}}) \longrightarrow\cdot\cdot\cdot
$$

In particular, the map $f_1$ is injective. We then choose another $T_2\subset S$ corresponding to an empty $2$-cell $L_{T_2}$ in $L$ and denote by $CL_{T_2}$ the right-angled cone on $L_{T_2}$. By a similar argument, the map $f_2$ in the following Mayer--Vietoris sequence is injective:

$$\cdot\cdot\cdot \longrightarrow L_\mathbf{1}^2H_2(\Sigma_{CL_{T_2}}) \oplus L_\mathbf{1}^2H_2(\Sigma_{L\cup CL_{T_2}})\overset{f_2}{\longrightarrow} L_\mathbf{1}^2H_2(\Sigma_{L\cup CL_{T_1}\cup CL_{T_2}}) \longrightarrow\cdot\cdot\cdot
$$

Proceed inductively until all empty $2$-cells have been coned off and denote the newly promoted nerve by $L'$. The $f_i$'s yield a sequence of injective maps:

$$\begin{tikzcd}[column sep=small, row sep=small]
 L_\mathbf{1}^2H_2(\Sigma_L) \arrow[hookrightarrow]{r} & L_\mathbf{1}^2H_2(\Sigma_{L\cup CL_{T_1}}) \arrow[hookrightarrow]{r} & \cdot\cdot\cdot\arrow[hookrightarrow]{r} & L_\mathbf{1}^2H_2(\Sigma_{L'})
\end{tikzcd}$$

Since $L'$ is a triangulation of $S^2$, it follows that $\Sigma_{L'}$ is a 3-manifold. Now, a result of Lott and Luck \cite{LottLuck}, in conjunction with the validity of the Geometrization Conjecture for $3$-manifolds \cite{Perelman}, implies that $L_\mathbf{1}^2H_\ast(\Sigma_{L'})$ vanishes in all dimensions. In particular, $L_\mathbf{1}^2b_2(\Sigma_L)=0$.
\end{proof}

\begin{remark}
Schroeder proves a more general theorem for $\Q=\mathbf{1}$ \cite[Theorem 4.6]{Schroeder1}. A metric flag complex $L$ is \emph{planar} if it can be embedded as a proper subcomplex of a triangulation of the $2$-sphere. Schroeder proves that if the nerve $L$ of a Coxeter system is planar, then $L_\mathbf{1}^2b_k(\Sigma_L)=0$ for $k\geq 2$. If $L$ is planar and $W$ is the corresponding Coxeter group, then \cite[Corollary 8.5.5]{Davis} implies that $\vcd W\leq 2$. Therefore we can use Lemma \ref{lemma:pushingupcycles} to deduce that $L_\Q^2b_k(\Sigma_L)=0$ for $k\geq 2$ and $\Q\leq \mathbf{1}$.
\end{remark}

\begin{theorem}
\label{thm:l2cellulationghs}
Suppose that the labeled nerve $L$ is the one-skeleton of a cell complex that is a $GHS^n$, $n\geq 2$, where all $2$-cells are Euclidean, and let $(W,S)$ denote the corresponding Coxeter system. Then $L_\Q^2b_\ast(\Sigma_L)$ is concentrated in degree $2$.

Furthermore,
\begin{align*}
L_\Q^2b_2(\Sigma_L)&= 1-\sum_{s\in S}\frac{q_s}{1+q_s}+\sum_{\substack{s,t\in S\\ m_{st}=2}}\frac{q_sq_t}{1+q_s+q_t+q_sq_t} + \sum_{\substack{s,t\in S\\ m_{st}=3}}\frac{q_s^3}{1+2q_s+2q_s^2+q_s^3}+\\
&+\sum_{\substack{s,t\in S\\ m_{st}=4}}\frac{q_s^2q_t^2}{1+q_s+q_t+2q_sq_t+q_s^2q_t+q_sq_t^2+q_s^2q_t^2}+\\
&+\sum_{\substack{s,t\in S\\ m_{st}=6}}\frac{q_s^3q_t^3}{1+q_s+q_t+2q_sq_t+q_s^2q_t+q_sq_t^2+2q_s^2q_t^2+q_s^2q_t^3+q_s^3q_t^2+q_s^3q_t^3}.
\end{align*}
\end{theorem}

\begin{proof}
Proposition \ref{prop:betti0} implies that $L_\Q^2b_0(\Sigma_L)=0$. Proposition \ref{prop:eulerchar}, along with Theorem \ref{thm:growthseriesformula}, explicitly compute the formula for $L_\Q^2b_2(\Sigma_L)$. We now turn our attention to showing $L_\Q^2b_1(\Sigma_L)=0$.

We prove the theorem by induction on $n$. For the base case $n=2$, first note that for every $T\in\mathcal{N}_P$, $\sigma_T$ is Euclidean. Hence Proposition \ref{prop:boundarycomponentsareproducts} implies that each $C_T$ appearing in $\partial\fchamber$ corresponds to a set $T\in\mathcal{N}_P$ where $W_T$ is a Euclidean reflection group. Thus Corollary \ref{cor:l2boundarycomponents} and Theorem \ref{thm:weightedeuclidean} imply that $L_\Q^2b_1(\mathcal{U}(W,C_T))=0$. This and Proposition \ref{prop:betti0} imply that the $E_1^{0,1}$ and $E_1^{1,0}$ terms in the $E_1$ sheet of the spectral sequence in Proposition \ref{prop:specialspectralsequence} are zero, which in turn implies that $L_\Q^2b_1(\partial\fdavis_L)=0$. Now, note that $\fdavis_L$ is three-dimensional and $\vcd W=2$. Moreover, by Lemma \ref{lemma:l2spherecellulations}, $L_{\Q^{-1}}^2H_2(\Sigma_L)=0$. Therefore, via Lemma \ref{lemma:h1bdry} (i), we reach the conclusion that $L_\Q^2b_1(\Sigma_L)=0$.

Now, suppose the theorem is true for $m<n$. Since $\Sigma_L$ is two-dimensional, Lemma \ref{lemma:h1bdry} (ii) tells us that we are done if we show that $L_\Q^2b_1(\partial\fdavis_L)=0$. Let $T\in\mathcal{N}_P$. Then $\sigma_T$ is the $(\partial\sigma_T,T)$-chamber, where $\sigma_T$ is the geometric cell in $P^\ast$ spanned by $T$. In particular, $\partial\sigma_T$ is a cell complex that is $GHS^m$, $m<n$, and since all $2$-cells of $P^\ast$ are Euclidean, it follows that all $2$-cells of $\partial\sigma_T$ are Euclidean. Hence, by induction and Corollary \ref{cor:l2boundarycomponents}, it follows that for every $T\in\mathcal{N}_P$, $L_\Q^2b_1(\mathcal{U}(W,C_T))=L_\Q^2b_1(\Sigma_{L_T})=0$. This and Proposition \ref{prop:betti0} imply that the $E_1^{0,1}$ and $E_1^{1,0}$ terms in the $E_1$ sheet of the spectral sequence in Proposition \ref{prop:specialspectralsequence} are zero, which in turn implies that $L_\Q^2b_1(\partial\fdavis_L)=0$.
\end{proof}

Consider the special case of Theorem \ref{thm:l2cellulationghs} when $n=2$. In this case, Theorem \ref{thm:l2cellulationghs}, along with Lemma \ref{lemma:l2spherecellulations}, explicitly compute the $L^2_\Q$-Betti numbers for all $\Q$: they are always concentrated in a single dimension. We emphasize this in the following corollary.

\begin{cor}
\label{cor:l2cellulationghs2}
Suppose that the labeled nerve $L$ is the one-skeleton of a cell complex that is a $GHS^2$, where all $2$-cells are Euclidean.
\begin{itemize}
  \item If $\Q\in\bar{\mathcal{R}}$, then $L_\Q^2H_\ast(\Sigma_L)$ is concentrated in dimension 0.
  \item If $\Q\notin\mathcal{R}$ and $\Q\leq\mathbf{1}$, then $L_\Q^2H_\ast(\Sigma_L)$ is concentrated in dimension 1.
  \item If $\Q\geq\mathbf{1}$, then $L_\Q^2H_\ast(\Sigma_L)$ is concentrated in dimension 2.
\end{itemize}
Furthermore,
\begin{align*}
\chi_\Q(\Sigma_L)&= 1-\sum_{s\in S}\frac{q_s}{1+q_s}+\sum_{\substack{s,t\in S\\ m_{st}=2}}\frac{q_sq_t}{1+q_s+q_t+q_sq_t} + \sum_{\substack{s,t\in S\\ m_{st}=3}}\frac{q_s^3}{1+2q_s+2q_s^2+q_s^3}+\\
&+\sum_{\substack{s,t\in S\\ m_{st}=4}}\frac{q_s^2q_t^2}{1+q_s+q_t+2q_sq_t+q_s^2q_t+q_sq_t^2+q_s^2q_t^2}+\\
&+\sum_{\substack{s,t\in S\\ m_{st}=6}}\frac{q_s^3q_t^3}{1+q_s+q_t+2q_sq_t+q_s^2q_t+q_sq_t^2+2q_s^2q_t^2+q_s^2q_t^3+q_s^3q_t^2+q_s^3q_t^3}.
\end{align*}
\end{cor}

If we place some restrictions on either our labels or the cell complex, then the formulas in Theorem \ref{thm:l2cellulationghs} become relatively simple, as illustrated by the following corollaries.

\begin{cor}
Suppose that $L$ is the one-skeleton of a cell complex that is a $GHS^n$, $n\geq 2$, where all $2$-cells are $2$-simplices. Give $L$ the labels $m_{st}=3$. Then $L_q^2b_\ast(\Sigma_L)$ is concentrated in degree $2$.

Furthermore, $$L_q^2b_2(\Sigma_L)=1-\frac{Vq}{1+q}+\frac{Eq^3}{1+2q+2q^2+q^3},$$
where $V$ and $E$ are the number of vertices and edges of $L$, respectively.
\end{cor}

Recall that an $n$-dimensional octahedron has $2n$ vertices and $2n(n-1)$ edges.

\begin{cor}
Suppose that $L$ the one skeleton of an $n$-dimensional octahedron with $n\geq 3$ and the labels $m_{st}=3$. Then $L_\Q^2b_\ast(\Sigma_L)$ is concentrated in degree $2$. Furthermore,
$$L_q^2b_2(\Sigma_L)=1-\frac{2nq}{1+q}+\frac{2n(n-1)q^3}{(1+2q+2q^2+q^3)}.$$
\end{cor}

\begin{cor}
\label{cor:l2kn3}
Let $L=K_n(3)$ with $n\geq 3$. Then $L_\Q^2b_\ast(\Sigma_L)$ is concentrated in degree $2$. Furthermore,
$$L_q^2b_2(\Sigma_L)=1-\frac{nq}{1+q}+\frac{n(n-1)q^3}{2(1+2q+2q^2+q^3)}.$$
\end{cor}

\begin{remark}
Note that under the hypothesis of the above corollaries, all generators in $S$ are conjugate, so in this case $\Q=q$, where $q\geq 1$ is a positive real number.
\end{remark}

If we assume that $W$ is right-angled, we have the following consequences of Theorem \ref{thm:l2cellulationghs}.

\begin{cor}
Suppose that $L$ is the one-skeleton of a cell complex that is a $GHS^n$, $n\geq 2$, where all $2$-cells are $2$-cubes. Give $L$ the labels $m_{st}=2$. Then $L_\Q^2b_\ast(\Sigma_L)$ is concentrated in degree $2$.
Furthermore, $$L_\Q^2b_2(\Sigma_L)=1-\sum_{s\in S}\frac{q_s}{1+q_s}+\sum_{\{s,t\}\in\mathcal{S}}\frac{q_sq_t}{1+q_s+q_t+q_sq_t}.$$
\end{cor}

Analogous to the case where $L=K_n(3)$, let $C_n(2)$ denote the one-skeleton of an $n$-cube with edges labeled by 2. If we assume that $L=C_n(2)$ and that $\Q=q$, where $q$ is a positive real number, then we obtain simple formulas for the $L^2_\Q$-Betti numbers. Recall that an $n$-cube has $2^n$ vertices and $n2^{n-1}$ edges.

\begin{cor}
Let $L=C_n(2)$ with $n\geq 2$. Then $L_q^2b_\ast(\Sigma_L)$ is concentrated in degree $2$. Furthermore, $$L_q^2b_2(\Sigma_L)=1-\frac{2^nq}{1+q}+\frac{n2^{n-1}q^2}{1+2q+q^2}.$$
\end{cor}

We can also allow ourselves to remove some edges from $L=K_n(3)$. We denote by $K_n^l(3)$ the complete graph on $n$ vertices, labeled by $3$'s and with $l$ edges removed. We have the following consequence of Corollary \ref{cor:l2kn3}.

\begin{cor}
Suppose that $L=K_n^{l}(3)$, where $n\geq 5$ and $l\leq n-4$. Then $L_q^2b_\ast(\Sigma_L)$ is concentrated in degree $2$. Furthermore,
$$L_q^2b_2(\Sigma_L)=1-\frac{nq}{1+q}+\frac{n(n-1)q^3}{2(1+2q+2q^2+q^3)}-\frac{lq^3}{1+2q+2q^2+q^3}.$$
\end{cor}
\begin{proof}
We first note that removing an edge from $K_n(3)$ splits the graph into two copies of $K_{n-1}(3)$ intersecting at $K_{n-2}(3)$. Since $n\geq 5$ and $q\geq 1$ we have the following Mayer--Vietoris sequence:

\[\tag{$\star$}\begin{tikzcd}[column sep=tiny]
\cdot\cdot\cdot \arrow{r} & L_q^2H_1(\Sigma_{K_{n-2}}) \arrow{r} & L_q^2H_1(\Sigma_{K_{n-1}})\oplus L_q^2H_1(\Sigma_{K_{n-1}})\arrow{r} & L_q^2H_1(\Sigma_{K_n^1}) \arrow{r} & 0
\end{tikzcd}\]

We first handle the case where $L=K_5^1(3)$. Removing an edge from $K_5(3)$ splits the graph into two copies of $K_{4}(3)$ intersecting at $K_{3}(3)$. Corollary \ref{cor:l2kn3} computes the $L_q^2$-(co)homology of each of the pieces in this decomposition and applying the sequence ($\star$) now proves the assertion for the case $L=K_5^1(3)$.

The proof for $L=K_n^{l}(3)$ is now by induction, the above computation serving as the base case. Suppose that the theorem is true for $m<n$. Begin by removing an edge from $K_n(3)$, splitting it as two copies of $K_{n-1}(3)$ intersecting at $K_{n-2}(3)$. Now, we remove the remaining $l-1\leq n-5$ edges from each of the graphs in the splitting, the worst case scenario being that we remove $l-1$ edges from $K_{n-2}(3)$ (which in turn removes $l-1$ edges from each copy of $K_{n-1}(3)$). Nevertheless, the inductive hypothesis is satisfied for each of the $K_{n-1}$'s in the splitting no matter how the remaining edges are removed. Applying a Mayer--Vietoris sequence analogous to ($\star$) now shows that the theorem holds for $L=K_n^{l}(3)$.
\end{proof}

With the help of some special subcomplexes of $\Sigma_L$ defined in \cite[Section 6]{DDJO}, we are also able to make computations when we change some labels on $L=K_n(3)$.

As before, $\Sigma_{cc}$ is $\Sigma_L$ with the Coxeter cellulation. Let $(W,S)$ be a Coxeter system and for $U\subset S$, set $\mathcal{S}(U):=\{T\in\mathcal{S}\mid T\subset U\}$. Define $\Sigma(U)$ to be the subcomplex of $\Sigma_{cc}$ consisting of all (closed) Coxeter cells of type $T$ with $T\in\mathcal{S}(U)$. Given $T\in\mathcal{S}(U)$, we define the following subcomplexes of $\Sigma(U)$:

\begin{align*}
\Omega_{UT}: & \hskip2mm \text{the union of closed cells of type }T', \text{ with }T'\in\mathcal{S}(U)_{\geq T},\\
\partial\Omega_{UT}: & \hskip2mm \text{the cells of }\Omega(U,T)\text{ of type } T'', \text{ with } T''\not\in\mathcal{S}(U)_{\geq T}.
\end{align*}

The pair $(\Omega_{UT},\partial\Omega_{UT})$ is the \emph{$(U,T)$-ruin}. For brevity, we write $(\Omega_{UT},\partial)$. For $s\in T$, set $U'=U-s$ and $T'=T-s$. As in \cite[Proof of Theorem 8.3]{DDJO}, we have the following weak exact sequence

\[\tag{$\dagger$}\begin{tikzcd}[column sep=small]
\cdot\cdot\cdot \arrow{r} & L_\Q^2H_{\ast}(\Omega_{U'T'},\partial) \arrow{r} & L_\Q^2H_{\ast}(\Omega_{UT'},\partial) \arrow{r} & L_\Q^2H_{\ast}(\Omega_{UT},\partial) \arrow{r} &\cdot\cdot\cdot
\end{tikzcd}\]

For the special case $U=S$ and $T=\{s\}$, the sequence becomes:

\[\tag{$\dagger\dagger$}\begin{tikzcd}[column sep=small]
\cdot\cdot\cdot \arrow{r} & L_\Q^2H_{\ast}(\Sigma(S-s)) \arrow{r} & L_\Q^2H_{\ast}(\Sigma(S)) \arrow{r} & L_\Q^2H_{\ast}(\Omega_{S\{s\}},\partial) \arrow{r} &\cdot\cdot\cdot
\end{tikzcd}\]

\begin{theorem}
Let $L=K_n$, the complete graph on $n$ vertices with $n\geq 5$. Let $k\leq n-4$, and suppose that we label $k$ edges of $L$ with $m_{st}\in\mathbb{N}-\{1,3\}$ and label the remaining edges by $3$. Then $L_\Q^2b_\ast(\Sigma_L)$ is concentrated in degree $2$.
\end{theorem}
\begin{proof}
The proof is by induction on $n$. First consider the case where $L=K_5$ with one label $m_{st}\in\mathbb{N}-\{1,3\}$. Then by Corollary \ref{cor:l2kn3}, $L^2b_1(\Sigma(S-s))=L^2b_1(\Sigma_{K_4(3)})=0$. According to sequence ($\dagger\dagger$), it remains to show that $L_\Q^2H_1(\Omega_{S\{s\}},\partial)=0$. We turn our attention to sequence ($\dagger$) with $U=S$, $T=\{s,t\}$, $U'=S-t$, and $T'=\{s\}$. By \cite[Lemma 8.1]{DDJO} $L_\Q^2H_1(\Omega_{ST},\partial)=0$, the point being that the relative chain complex of $(\Omega_{ST},\partial \Omega_{ST})$ has no one-dimensional cells. So, by weak exactness, it remains to show that $L_\Q^2H_1(\Omega_{U'T'},\partial)=0$. We consider the following version of sequence $(\dagger\dagger)$:

$$\begin{tikzcd}[column sep=small]
\cdot\cdot\cdot \arrow{r} & L_\Q^2H_1(\Sigma(S-\{s,t\})) \arrow{r} & L_\Q^2H_1(\Sigma(S-t)) \arrow{r} & L_\Q^2H_1(\Omega_{U'T'},\partial) \arrow{r} &\cdot\cdot\cdot
\end{tikzcd}$$

Note that $$L_\Q^2b_0(\Sigma(S-\{s,t\}))=L_\Q^2b_0(\Sigma_{K_3(3)})=0$$ and $$L_\Q^2b_1(\Sigma(S-t))=L_\Q^2b_1(\Sigma_{K_4(3)})=0$$ by Theorem \ref{thm:weightedeuclidean} and Corollary \ref{cor:l2kn3}, respectively. By weak exactness, $L_\Q^2H_1(\Omega_{U'T'},\partial)=0$, and hence $L_\Q^2H_1(\Omega_{S\{s\}},\partial)=0$, thus proving the assertion for $L=K_5$.

Now, suppose that the theorem is true for $L=K_m$, $m<n$. We wish to show the theorem is true for $L=K_n$. Begin by choosing an edge $e$ with vertices $s$ and $t$ and label different from $3$. We now observe that $L^2b_1(\Sigma(S-s))=L^2b_1(\Sigma_{K_{n-1}})=0$ by the inductive hypothesis, since $K_{n-1}$ now has at most $n-5$ edges with a label different from $3$. Similarly, the inductive hypothesis implies $L_\Q^2b_1(\Sigma(S-t))=0$ and $L_\Q^2b_0(\Sigma(S-\{s,t\}))=0$. Hence the weak exact sequences used in the proof for the case $L=K_5$ allow us to conclude that $L_\Q^2b_1(\Sigma_L)=L_\Q^2b_1(\Sigma(S))=0$.
\end{proof}

\begin{remark}
Note that in conjunction with Theorem \ref{thm:weightedeuclidean} and Corollary \ref{cor:l2cellulationghs2}, the above argument gives an alternate proof of Corollary \ref{cor:l2kn3}.
\end{remark}

\subsection{The Weighted Singer Conjecture.}
Appearing in \cite{DDJO}, the following is the appropriate formulation of the the Singer Conjecture for Coxeter groups \cite{DO} for weighted $L^2$-(co)homology:

\begin{conjecture}[Weighted Singer Conjecture]
\label{conj:Singer}
Suppose that the nerve $L$ is a triangulation of $S^{n-1}$. Then $$L_\Q^2H_k(\Sigma_L)=0 \text{ for } k>\frac{n}{2} \text{ and } \Q\leq \mathbf{1}.$$
\end{conjecture}

By weighted Poincar\'{e} duality, this is equivalent to the conjecture that if $\Q\geq \mathbf{1}$ and $k<\frac{n}{2}$, then $L_\Q^2H_k(\Sigma)$ vanishes. The conjecture is known for elementary reasons for $n\leq 2$, and in \cite{DDJO}, it is proved for the case where $W$ is right-angled and $n\leq 4$. Furthermore, it was shown in in \cite{DDJO} that Conjecture \ref{conj:Singer} for $n$ odd implies Conjecture \ref{conj:Singer} for $n$ even, under the assumption that $W$ is right-angled.

The original Singer conjecture for Coxeter groups was formulated for $\Q=\mathbf{1}$ in \cite{DO} and concluded that the $L^2$-(co)homology is concentrated in dimension $\frac{n}{2}$. The original conjecture is known for elementary reasons for $n\leq 2$ and holds by a result of Lott and L\"{u}ck \cite{LottLuck}, in conjunction with the validity of the Geometrization Conjecture for $3$-manifolds \cite{Perelman}, for $n=3$. It was proved by Davis--Okun \cite{DO} for the case where $W$ is right-angled and $n\leq 4$. It was later proved for the case where $W$ is an even Coxeter group and $n\leq 4$ by Schroeder \cite{Schroeder}, under the assumption that the nerve $L$ is a flag complex. Due to recent work of Okun--Schreve \cite[Theorem 4.9]{OS}, the conjecture is now known in full generality whenever $\Q=\mathbf{1}$ and $n\leq 4$. In fact, using induction and \cite[Theorem 4.5, Lemma 4.6, Corollary 4.7]{OS} proves the following theorem.

\begin{theorem}
\label{thm:singerq=1}
Suppose that the nerve $L$ is an $(n-1)$-sphere or an $(n-1)$-disk. Then $$L_\mathbf{1}^2H_k(\Sigma_L)=0 \text{ for } k\geq n-1.$$
\end{theorem}

Note that if $L$ is a triangulation of the $(n-1)$-disk, then $\Sigma_L$ is an $n$-manifold with boundary. We now obtain the following theorem, which whenever $n=3,4$ can be thought of as a version of Conjecture \ref{conj:Singer} for the case where $\Sigma_L$ is an $n$-manifold with boundary.

\begin{theorem}
\label{thm:singerfordisks}
Suppose that the nerve $L$ is an $(n-1)$-disk. Then $$L_\Q^2H_k(\Sigma_L)=0 \text{ for } k\geq n-1 \text{ and } \Q\leq \mathbf{1}.$$
\end{theorem}
\begin{proof}
By Theorem \ref{thm:singerq=1}, we have that $L_\mathbf{1}^2H_k(\Sigma_L)=0$ for $k\geq n-1$. Furthermore, \cite[Corollary 8.5.5]{Davis} implies that $\vcd W\leq n-1$, and hence we are done by Lemma \ref{lemma:pushingupcycles}.
\end{proof}

We note that Theorem \ref{thm:l2cellulationghs} provides convincing evidence for the validity of a weighted version of Theorem \ref{thm:singerq=1} when $L$ is a triangulation of the $(n-1)$-sphere. Suppose that the labeled nerve $L'$ is the one-skeleton of a cellulation of a $GHS^{n-1}$, $n\geq 3$, where all $2$-cells are Euclidean. Build $L'$ to a triangulation that is a $GHS^{n-1}$ by coning on each empty cell and labeling new edges by $2$. In other words, perform the following sequence of right-angled cones. First begin by coning on each empty $2$-cell, then on each empty $3$-cell, and so on, until each empty cell has been coned off. (if $n=3$, this process stops when each empty $2$-cell has been coned off).

\begin{theorem}
\label{thm:singerghs}
Suppose that the nerve $L$ a $GHS^{n-1}$, $n\geq 3$, obtained via the above construction and suppose that $\Q\geq \mathbf{1}$. Then $$L_\Q^2b_k(\Sigma_{L})=0 \text{ for } k\leq 1.$$
\end{theorem}
\begin{proof}
The proof of the theorem follows the strategy of Lemma \ref{lemma:l2spherecellulations}: one performs careful book-keeping using Mayer--Vietoris sequences when constructing $L$ from $L'$. Theorem \ref{thm:l2cellulationghs} tells us that $L'$ originally satisfies $L_\Q^2b_1(\Sigma_{L'})=0$. To construct $L$ from $L'$, we first began by coning empty $2$-cells, then successively coning higher dimensional cells, labeling new edges by $2$. If at each step of this process we employ a Mayer--Vietoris sequence, then Theorem \ref{thm:l2cellulationghs}, in conjunction with the fact that right-angled cones will not develop new homology below dimension $2$, implies that $L_\Q^2b_1(\Sigma_{L})=0$.
\end{proof}

\subsection{Quasi-L\'{a}nner groups.} A $2$-spherical Coxeter group $W$ is \emph{quasi-L\'{a}nner} if it acts properly (but not cocompactly) on hyperbolic space $\mathbb{H}^n$ by reflections with fundamental chamber an $n$-simplex of finite volume. For brevity, we say that $W$ is of type $QL_n$. Quasi-L\'{a}nner groups have been classified and only exist in dimensions $3$ through $10$. For a complete list, see \cite[$\S$6.9]{Humphreys}. We note that the Coxeter group with corresponding nerve $L=K_4(3)$ is on the list.

All non-spherical proper special subgroups of a quasi-L\'{a}nner group are Euclidean and on the list appearing in \cite[pg. 34]{Humphreys}. Moreover, if $W$ is of type $QL_n$, then the only proper infinite special subgroups are those $W_T$ with $|T|=n-1$. Hence, by \cite[Corollary 8.5.5]{Davis}, if $W$ is of type $QL_n$, then $\vcd W=n-1$. With this observation, we prove the following theorem.

\begin{theorem}
\label{thm:l2QL}
Suppose that $W$ is of type $QL_n$. Then $L_\Q^2b_k(\Sigma_L)=0$ whenever $k\geq n-1$ and $\Q\leq\mathbf{1}$, or $k\leq 1$ and $\Q\geq\mathbf{1}$.
\end{theorem}
\begin{proof}
We first suppose that $\Q= \mathbf{1}$. Since $W$ is of type $QL_n$, we can realize a finite volume $n$-simplex in hyperbolic space $\mathbb{H}^n$, with $W$ acting by reflections along codimension-one faces (note that this simplex has some ideal vertices). By a theorem of Cheeger-Gromov \cite{CheegerGromov}, $L_{\mathbf{1}}^2H_k(\Sigma_L)\cong L_{\mathbf{1}}^2\mathcal{H}^k(\mathbb{H}^n)$, where $L_{\mathbf{1}}^2\mathcal{H}^k$ denotes the $L^2$ de Rham cohomology. By a theorem of Dodziuk \cite{Dodziuk}, $L_{\mathbf{1}}^2\mathcal{H}^k(\mathbb{H}^n)=0$ for all $k\geq 0$ if $n$ is odd, and is concentrated in dimension $\frac{n}{2}$ if $n$ is even. In particular, $L_{\mathbf{1}}^2b_{n-1}(\Sigma_L)=0$. The result for $\Q\leq\mathbf{1}$ now follows by Lemma \ref{lemma:pushingupcycles} and the fact that $\vcd W=n-1$.

Now, suppose that $\Q\geq \mathbf{1}$. Consider the fattened Davis complex $\fdavis_L$ with respect to $P=\Delta^n$, the standard $n$-simplex (see Remark \ref{remark:simplexpolytope} and Figure \ref{figure:truncatedtetrahedron}).

\begin{figure}[h]
\[\begin{tikzpicture}[scale=1.6]
\tikzstyle{vertex}=[circle, fill=black, draw, inner sep=0pt, minimum size=1pt]
\tikzstyle{every node}=[draw,circle,fill=white,minimum size=0pt,
                            inner sep=0pt]
\begin{scope}[rotate=-23]

\path[transparent]
    (1,1,1) edge node[pos=0.2] (r1) {} node[pos=0.8] (s1) {} (1,-1,-1)
    (1,1,1) edge node[pos=0.2] (r3) {} node[pos=0.8] (u1) {} (-1,1,-1)
    (1,-1,-1) edge node[pos=0.2] (s2) {} node[pos=0.8] (t2) {} (-1,-1,1)
    (-1,-1,1) edge node[pos=0.2] (t3) {} node[pos=0.8] (u3) {} (-1,1,-1)
    (1,-1,-1) edge node[pos=0.2] (s3) {} node[pos=0.8] (u2) {} (-1,1,-1)
    (1,1,1) edge node[pos=0.2] (r2) {} node[pos=0.8] (t1) {} (-1,-1,1)
    ;

\path
    (r3) edge (r1)
    (s1) edge (s2)
    (s2) edge (s3)
    (s3) edge (s1)
    (u1) edge (u2)
    (u2) edge (u3)
    (u3) edge (u1)
    (t2) edge (t3)

    ;

\path[gray, dashed]
    (r1) edge (r2)
    (r2) edge (r3)
    (t1) edge (t2)
    (t3) edge (t1)
    ;

\path
    (s3) edge (u2)
    (t3) edge (u3)
    (s2) edge (t2)
    (r3) edge (u1)
    (r1) edge (s1)
;

\path[gray, dashed]
    (r2) edge (t1)
;
\end{scope}
\end{tikzpicture}\]
\caption{:\text{ }$\fchamber$ when $L=K_4(3)$}
\label{figure:truncatedtetrahedron}
\end{figure}
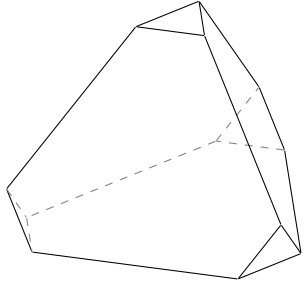

Weighted Poincar\'{e} duality implies that $$L_\Q^2H_1(\fdavis_L,\partial\fdavis_L)\cong L_{\Q^{-1}}^2H_{n-1}(\fdavis_L)\cong L_{\Q^{-1}}^2H_{n-1}(\Sigma_L)=0,$$ so by the long exact sequence for the pair $(\fdavis_L,\partial\fdavis_L)$ it remains to show $L_\Q^2H_1(\partial\fdavis_L)=0$. Proposition \ref{prop:boundaryfdavis} implies that each $C_T$ appearing in $\partial\fchamber$ corresponds to a set $T\in\mathcal{N}_P$ with $W_T$ a Euclidean reflection group. In particular, Corollary \ref{cor:l2boundarycomponents} and Theorem \ref{thm:weightedeuclidean} imply that $L_\Q^2b_1(\mathcal{U}(W,C_T))=0$. Hence the $E_1^{0,1}$ term in the $E_1$ sheet of the spectral sequence of Proposition \ref{prop:specialspectralsequence} is zero. By Proposition \ref{prop:betti0}, the first row of the $E_1$ sheet is also zero, and in particular $E_1^{1,0}$ is zero. Therefore $L_\Q^2b_1(\partial\fdavis_L)=0$.
\end{proof}

Of important note is the case when $W$ is $QL_3$. In this special case, Theorem \ref{thm:l2QL} explicitly computes the $L^2_\Q$-Betti numbers for all $\Q$: they are always concentrated in a single dimension.

\begin{cor}
Suppose that $W$ is of type $QL_3$. Then
\begin{itemize}
  \item If $\Q\in\bar{\mathcal{R}}$, then $L_\Q^2H_\ast(\Sigma_L)$ is concentrated in dimension 0.
  \item If $\Q\notin\mathcal{R}$ and $\Q\leq\mathbf{1}$, then $L_\Q^2H_\ast(\Sigma_L)$ is concentrated in dimension 1.
  \item If $\Q\geq\mathbf{1}$, then $L_\Q^2H_\ast(\Sigma_L)$ is concentrated in dimension 2.
\end{itemize}
\end{cor}

Since the $L^2_\Q$-(co)homology is always concentrated in a single dimension, one can use Proposition \ref{prop:eulerchar}, along with Theorem \ref{thm:growthseriesformula}, to obtain explicit formulas for the $L^2_\Q$-Betti numbers.

\subsection{Other $2$-spherical groups.} We now perform computations for other $2$-spherical groups, removing the restriction that the nerve $L$ is a graph. Given a Coxeter system $(W,S)$, we make a particular choice of $P$ for the construction of $\fdavis_L$, namely $P=\Delta^{|S|-1}$, the standard $(|S|-1)$-simplex (see Remark \ref{remark:simplexpolytope}).

While one could argue the following lemma using the spectral sequence, we use a simple Mayer--Vietoris sequence argument to illustrate the technique behind the machinery.

\begin{lemma}
\label{lemma:l2k5}
Suppose that $(W,S)$ is infinite $2$-spherical with $|S|=5$ and $\vcd W\leq 3$. Furthermore, suppose that every infinite special subgroup $W_T$, with $|T|=3,4$, is Euclidean or $QL_3$, and that $L_{\mathbf{1}}^2b_3(\Sigma_L)=0$. Then $L_\Q^2b_\ast(\Sigma_L)$ is concentrated in degree $2$.
\end{lemma}
\begin{proof}
We wish to reduce the proof to showing that $L_\Q^2b_1(\partial\fdavis_L)=0$. If $\vcd W=2$, then this is accomplished by Lemma \ref{lemma:h1bdry} (ii). If $\vcd W=3$, then according to Lemma \ref{lemma:h1bdry} (i), we reduce the proof to showing $L_\Q^2b_1(\partial\fdavis_L)=0$ if we show that $L_{\Q^{-1}}^2b_3(\Sigma_L)=0$. By Lemma \ref{lemma:pushingupcycles}, we reach this conclusion since by assumption $L_{\mathbf{1}}^2b_3(\Sigma_L)=0$. So, to complete the proof, we must show that $L_\Q^2b_1(\partial\fdavis_L)=0$.

Let $\mathcal{N}_P^{(j)}=\{T\in \mathcal{N}_P\mid\text{Card}(T)=j\}$ and set $$A_j=\bigsqcup_{T\in\mathcal{N}_P^{(j)}}\mathcal{U}(W,C_T).$$

Note that $|S|=5$ and all proper non-spherical subsets $T$ have order $3$ or $4$, so by Proposition \ref{prop:boundaryfdavis}, $\partial\fdavis_L=A_3\cup A_4$. Figure \ref{figure:boundaryk5} illustrates the chamber for $\partial\fdavis_L$ for the case where $L=K_5(3)$.

\begin{figure}[h!]
\[\begin{tikzpicture}
\tikzstyle{vertex}=[circle, draw, inner sep=0pt, minimum size=0pt]
\tikzstyle{every node}=[circle,inner sep=0pt,minimum size=0pt]

	\vertex (v) at (.5,-.5,0) {};
	\vertex (x) at (.5,3,0) {};
	\vertex (y) at (4,-3,.5) {};
    \vertex (z) at (.5,-3,5) {};
    \vertex (w) at (-2.5,-3,-2.5) {};

\path[white]
    (v) edge node[pos=.1] (v1) {} node[pos=.6] (x1) {} (x)
    (v) edge node[pos=.1] (v2) {} node[pos=.75] (y1) {} (y)
    (v) edge node[pos=.1] (v3) {} node[pos=.8] (z1) {} (z)
    (v) edge node[pos=.1] (v4) {} node[pos=.7] (w1) {} (w)
    (x) edge node[pos=.1] (x3) {} node[pos=.95] (y2) {} (y)
    (x) edge node[pos=.1] (x2) {} node[pos=.9] (z2) {} (z)
    (x) edge node[pos=.1] (x4) {} node[pos=.9] (w2) {} (w)
    (y) edge node[pos=.1] (y3) {} node[pos=.92] (z3) {} (z)
    (y) edge node[pos=.1] (y4) {} node[pos=.92] (w3) {} (w)
    (z) edge node[pos=.1] (z4) {} node[pos=.92] (w4) {} (w)
    ;

\path
    (v1) edge node[pos=.2] (v11) {} node[pos=.8] (v21) {} (v2)
    (v1) edge node[pos=.2] (v12) {} node[pos=.8] (v31) {} (v3)
    (v1) edge node[pos=.2] (v13) {} node[pos=.8] (v41) {} (v4)
    (v2) edge node[pos=.2] (v22) {} node[pos=.8] (v32) {} (v3)
    (v2) edge node[pos=.2] (v23) {} node[pos=.8] (v42) {} (v4)
    (v3) edge node[pos=.2] (v33) {} node[pos=.8] (v43) {} (v4)
    ;

\path[draw, line width=1pt, white] (v1) -- (v11);
\path[draw, line width=1pt, white] (v1) -- (v12);
\path[draw, line width=1pt, white] (v1) -- (v13);
\path[draw, line width=1pt, white] (v2) -- (v21);
\path[draw, line width=1pt, white] (v2) -- (v22);
\path[draw, line width=1pt, white] (v2) -- (v23);
\path[draw, line width=1pt, white] (v3) -- (v31);
\path[draw, line width=1pt, white] (v3) -- (v32);
\path[draw, line width=1pt, white] (v3) -- (v33);
\path[draw, line width=1pt, white] (v4) -- (v41);
\path[draw, line width=1pt, white] (v4) -- (v42);
\path[draw, line width=1pt, white] (v4) -- (v43);

\path
    (v11) edge (v12)
    (v12) edge (v13)
    (v13) edge (v11)
    (v21) edge (v22)
    (v22) edge (v23)
    (v23) edge (v21)
    (v31) edge (v32)
    (v32) edge (v33)
    (v33) edge (v31)
    (v41) edge (v42)
    (v42) edge (v43)
    (v43) edge (v41)
    ;

\path
    (x1) edge node[pos=.2] (x11) {} node[pos=.8] (x21) {} (x2)
    (x1) edge node[pos=.2] (x12) {} node[pos=.8] (x31) {} (x3)
    (x1) edge node[pos=.2] (x13) {} node[pos=.8] (x41) {} (x4)
    (x2) edge node[pos=.2] (x22) {} node[pos=.8] (x32) {} (x3)
    (x2) edge node[pos=.2] (x23) {} node[pos=.8] (x42) {} (x4)
    (x3) edge node[pos=.2] (x33) {} node[pos=.8] (x43) {} (x4)
    ;

\path[draw, line width=1pt, white] (x1) -- (x11);
\path[draw, line width=1pt, white] (x1) -- (x12);
\path[draw, line width=1pt, white] (x1) -- (x13);
\path[draw, line width=1pt, white] (x2) -- (x21);
\path[draw, line width=1pt, white] (x2) -- (x22);
\path[draw, line width=1pt, white] (x2) -- (x23);
\path[draw, line width=1pt, white] (x3) -- (x31);
\path[draw, line width=1pt, white] (x3) -- (x32);
\path[draw, line width=1pt, white] (x3) -- (x33);
\path[draw, line width=1pt, white] (x4) -- (x41);
\path[draw, line width=1pt, white] (x4) -- (x42);
\path[draw, line width=1pt, white] (x4) -- (x43);

\path
    (x11) edge (x12)
    (x12) edge (x13)
    (x13) edge (x11)
    (x21) edge (x22)
    (x22) edge (x23)
    (x23) edge (x21)
    (x31) edge (x32)
    (x32) edge (x33)
    (x33) edge (x31)
    (x41) edge (x42)
    (x42) edge (x43)
    (x43) edge (x41)
    ;

\path
    (x11) edge (v12)
    (x12) edge (v11)
    (x13) edge (v13)
;

\path
    (y1) edge node[pos=.2] (y11) {} node[pos=.8] (y21) {} (y2)
    (y1) edge node[pos=.2] (y12) {} node[pos=.8] (y31) {} (y3)
    (y1) edge node[pos=.2] (y13) {} node[pos=.8] (y41) {}(y4)
    (y2) edge node[pos=.2] (y22) {} node[pos=.8] (y32) {} (y3)
    (y2) edge node[pos=.2] (y23) {} node[pos=.8] (y42) {} (y4)
    (y3) edge node[pos=.2] (y33) {} node[pos=.8] (y43) {} (y4)
    ;

\path[draw, line width=1pt, white] (y1) -- (y11);
\path[draw, line width=1pt, white] (y1) -- (y12);
\path[draw, line width=1pt, white] (y1) -- (y13);
\path[draw, line width=1pt, white] (y2) -- (y21);
\path[draw, line width=1pt, white] (y2) -- (y22);
\path[draw, line width=1pt, white] (y2) -- (y23);
\path[draw, line width=1pt, white] (y3) -- (y31);
\path[draw, line width=1pt, white] (y3) -- (y32);
\path[draw, line width=1pt, white] (y3) -- (y33);
\path[draw, line width=1pt, white] (y4) -- (y41);
\path[draw, line width=1pt, white] (y4) -- (y42);
\path[draw, line width=1pt, white] (y4) -- (y43);

\path
    (y11) edge (y12)
    (y12) edge (y13)
    (y13) edge (y11)
    (y21) edge (y22)
    (y22) edge (y23)
    (y23) edge (y21)
    (y31) edge (y32)
    (y32) edge (y33)
    (y33) edge (y31)
    (y41) edge (y42)
    (y42) edge (y43)
    (y43) edge (y41)
    ;

\path
    (x31) edge [out=-20,in=60] (y21)
    (x32) edge [out=-20,in=60] (y22)
    (x33) edge [out=-20,in=60] (y23)
    ;

\path
    (v21) edge (y11)
    (v22) edge (y12)
    (v23) edge (y13)
    ;
\path
    (z1) edge node[pos=.2] (z11) {} node[pos=.8] (z21) {} (z2)
    (z1) edge node[pos=.2] (z12) {} node[pos=.8] (z31) {} (z3)
    (z1) edge node[pos=.2] (z13) {} node[pos=.8] (z41) {} (z4)
    (z2) edge node[pos=.2] (z22) {} node[pos=.8] (z32) {} (z3)
    (z2) edge node[pos=.2] (z23) {} node[pos=.8] (z42) {} (z4)
    (z3) edge node[pos=.2] (z33) {} node[pos=.8] (z43) {} (z4)
    ;

\path[draw, line width=1pt, white] (z1) -- (z11);
\path[draw, line width=1pt, white] (z1) -- (z12);
\path[draw, line width=1pt, white] (z1) -- (z13);
\path[draw, line width=1pt, white] (z2) -- (z21);
\path[draw, line width=1pt, white] (z2) -- (z22);
\path[draw, line width=1pt, white] (z2) -- (z23);
\path[draw, line width=1pt, white] (z3) -- (z31);
\path[draw, line width=1pt, white] (z3) -- (z32);
\path[draw, line width=1pt, white] (z3) -- (z33);
\path[draw, line width=1pt, white] (z4) -- (z41);
\path[draw, line width=1pt, white] (z4) -- (z42);
\path[draw, line width=1pt, white] (z4) -- (z43);

\path
    (z11) edge (z12)
    (z12) edge (z13)
    (z13) edge (z11)
    (z21) edge (z22)
    (z22) edge (z23)
    (z23) edge (z21)
    (z31) edge (z32)
    (z32) edge (z33)
    (z33) edge (z31)
    (z41) edge (z42)
    (z42) edge (z43)
    (z43) edge (z41)
    ;

\path
    (y31) edge [out=-100,in=-10] (z31)
    (y32) edge [out=-100,in=-10] (z32)
    (y33) edge [out=-100,in=-10] (z33)
    ;

\path
    (x21) edge [out=-130,in=90] (z21)
    (x22) edge [out=-130,in=90] (z22)
    (x23) edge [out=-130,in=90] (z23)
    ;

\path
    (v31) edge (z11)
    (v32) edge (z12)
    (v33) edge (z13)
    ;

\path
    (w1) edge node[pos=.2] (w11) {} node[pos=.8] (w21) {} (w2)
    (w1) edge node[pos=.2] (w12) {} node[pos=.8] (w31) {} (w3)
    (w1) edge node[pos=.2] (w13) {} node[pos=.8] (w41) {} (w4)
    (w2) edge node[pos=.2] (w22) {} node[pos=.8] (w32) {} (w3)
    (w2) edge node[pos=.2] (w23) {} node[pos=.8] (w42) {} (w4)
    (w3) edge node[pos=.2] (w33) {} node[pos=.8] (w43) {} (w4)
    ;

\path[draw, line width=1pt, white] (w1) -- (w11);
\path[draw, line width=1pt, white] (w1) -- (w12);
\path[draw, line width=1pt, white] (w1) -- (w13);
\path[draw, line width=1pt, white] (w2) -- (w21);
\path[draw, line width=1pt, white] (w2) -- (w22);
\path[draw, line width=1pt, white] (w2) -- (w23);
\path[draw, line width=1pt, white] (w3) -- (w31);
\path[draw, line width=1pt, white] (w3) -- (w32);
\path[draw, line width=1pt, white] (w3) -- (w33);
\path[draw, line width=1pt, white] (w4) -- (w41);
\path[draw, line width=1pt, white] (w4) -- (w42);
\path[draw, line width=1pt, white] (w4) -- (w43);

\path
    (w11) edge (w12)
    (w12) edge (w13)
    (w13) edge (w11)
    (w21) edge (w22)
    (w22) edge (w23)
    (w23) edge (w21)
    (w31) edge (w32)
    (w32) edge (w33)
    (w33) edge (w31)
    (w41) edge (w42)
    (w42) edge (w43)
    (w43) edge (w41)
    ;

\path
    (z41) edge [out=160,in=-150] (w41)
    (z42) edge [out=160,in=-150] (w42)
    (z43) edge [out=160,in=-150] (w43)
    ;

\path
    (y41) edge [out=140,in=20] (w31)
    (y42) edge [out=140,in=20] (w32)
    (y43) edge [out=140,in=20] (w33)
    ;

\path
    (x41) edge [out=-160,in=100] (w21)
    (x42) edge [out=-160,in=100] (w22)
    (x43) edge [out=-160,in=100] (w23)
    ;
\path
    (v41) edge (w11)
    (v42) edge (w12)
    (v43) edge (w13)
    ;

\end{tikzpicture}\]
\caption{: Fundamental chamber for $\partial\fdavis_L$ when $L=K_5(3)$}
\label{figure:boundaryk5}
\end{figure}
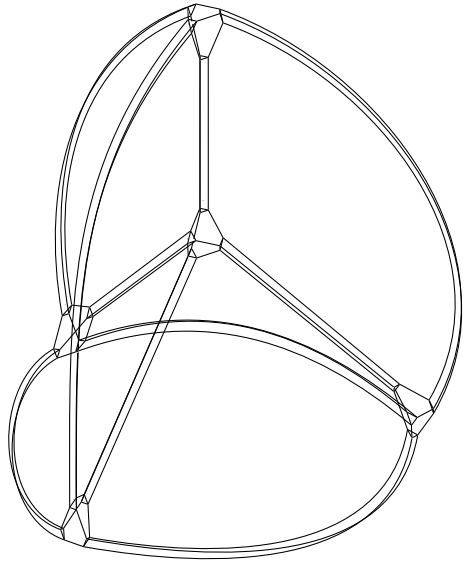

By Proposition \ref{prop:boundarycomponentsareproducts} (i),
\begin{align*}
A_3\cap A_4=\bigsqcup_{\substack{U\in\mathcal{N}_P^{(3)}\\ V\in\mathcal{N}_P^{(4)}\\U\subset V}}\mathcal{U}(W,C_U)\cap\mathcal{U}(W,C_V)
\end{align*}

By Corollary \ref{cor:l2boundarycomponents},
\begin{align*}
L_\Q^2b_k(A_j)&=\sum L_\Q^2b_k(\mathcal{U}(W,C_T))\\
&=\sum L_\Q^2b_k(\fdavis_{L_T})
\end{align*}

and
\begin{align*}
L_\Q^2b_k(A_3\cap A_4)&=\sum L_\Q^2b_k(\mathcal{U}(W,C_T))\\
&=\sum L_\Q^2b_k(\fdavis_{L_T})
\end{align*}

Here $L_T$ is the subcomplex of $L$ corresponding to the infinite subgroup $W_T$, which is either Euclidean or of type $QL_3$. By Theorem \ref{thm:weightedeuclidean} and Theorem \ref{thm:l2QL}, $L_\Q^2b_k(\fdavis_{L_T})=0$ for $k<2$. Hence \[\tag{$\diamond$}L_\Q^2H_k(A_j)=0 \text{ for } j=3,4 \text{ and } L_\Q^2H_k(A_3\cap A_4)=0 \text{ for } k<2. \]

Now, consider the Mayer--Vietoris sequence:
$$\begin{tikzcd}[column sep=small]
\cdot\cdot\cdot \arrow{r} & L_\Q^2H_k(A_3\cap A_4) \arrow{r} & L_\Q^2H_k(A_3)\oplus L_\Q^2H_k(A_4)\arrow{r} & L_\Q^2H_k(\partial\fdavis_L) \arrow{r} & \cdot\cdot\cdot
\end{tikzcd}$$

Inputting $(\diamond)$ into this sequence yields $$L_\Q^2H_k(\partial\fdavis_L)=0\text{ for } k<2.$$

Lemma \ref{lemma:h1bdry} now concludes that $L_\Q^2b_1(\Sigma_L)=0$.
\end{proof}

\begin{theorem}
\label{thm:genl2kn}
Suppose that $(W,S)$ is infinite $2$-spherical with $|S|\geq 5$. Suppose furthermore that:
\begin{enumerate}
  \item For every $T\subseteq S$ with $|T|\geq 5$, $\vcd W_T\leq |T|-2$.
  \item $L_\mathbf{1}^2b_{|S|-2}(\Sigma_L)=0$.
  \item Every infinite subgroup $W_T$, with $|T|=3,4$, is Euclidean or $QL_3$.
\end{enumerate}
Then $L_\Q^2b_k(\Sigma_{L})=0 \text{ for } k<2$.
\end{theorem}

\begin{proof}
The statement for $L_\Q^2b_0(\Sigma_L)$ follows from Proposition \ref{prop:betti0}. So, we turn our attention to showing $L_\Q^2b_1(\Sigma_L)=0$. The proof of the theorem is now by induction on $|S|$, Lemma \ref{lemma:l2k5} serving as the base case. By Lemma \ref{lemma:pushingupcycles}, since $\vcd W\leq |S|-2$, it follows that $L_{\Q^{-1}}^2b_{|S|-2}(\Sigma_L)=0$. Furthermore, $\fdavis_L$ has dimension $|S|-1$, so by Lemma \ref{lemma:h1bdry} it now suffices to show that $L_\Q^2b_1(\partial\fdavis_L)=0$. By assumption, every non-spherical special subgroup $W_U$ with $|U|=3,4$ is Euclidean or $QL_3$. Thus every non-spherical special subgroup $W_U$, with $4<|U|<|S|$ satisfies the inductive hypothesis. Therefore by induction, Theorem \ref{thm:weightedeuclidean}, and Theorem \ref{thm:l2QL}, for any $T\in\mathcal{N}_P$ we have that $L_\Q^2b_1(\Sigma_{L_T})=0$ (Here $L_T$ is the subcomplex of $L$ corresponding to the special subgroup $W_T$).

Hence by Corollary \ref{cor:l2boundarycomponents} (i), for every $T\in\mathcal{N}_P$

$$L_\Q^2b_1(\mathcal{U}(W,C_T))= L_\Q^2b_1(\Sigma_{L_T})=0.$$

This implies that the $E_1^{0,1}$ term in the $E_1$ sheet of the spectral sequence of Proposition \ref{prop:specialspectralsequence} is zero. By Proposition \ref{prop:betti0}, the first row of the $E_1$ sheet is also zero, and in particular $E_1^{1,0}$ is zero. Therefore $L_\Q^2b_1(\partial\fdavis_L)=0$.
\end{proof}

With the help of Theorem \ref{thm:singerq=1}, we drop condition $2$ in Theorem \ref{thm:genl2kn}.

\begin{cor}
\label{cor:genl2kn}
Suppose that $(W,S)$ is infinite $2$-spherical with $|S|\geq 5$. Suppose furthermore that:
\begin{enumerate}
  \item For every $T\subseteq S$ with $|T|\geq 5$, $\vcd W_T\leq |T|-2$.
  \item Every infinite subgroup $W_T$, with $|T|=3,4$, is Euclidean or $QL_3$.
\end{enumerate}
Then $L_\Q^2b_k(\Sigma_{L})=0 \text{ for } k<2$.
\end{cor}

\begin{proof}
Note that condition $2$ in Theorem \ref{thm:genl2kn} is vacuously satisfied if $\vcd W\leq |S|-3$, so we suppose that $\vcd W=|S|-2$. We must show that $L_\mathbf{1}^2b_{|S|-2}(\Sigma_L)=0$, and to do this, we use an argument analogous to the one in Lemma \ref{lemma:l2spherecellulations}. We first begin by coning empty $2$-simplices of $L$, and then empty $3$-simplices, and so on, until all empty simplices have been coned off. We then label all new edges by $2$. In this way we obtain a newly promoted nerve $L'$ which is a triangulation of $S^{|S|-2}$, and in particular, $\Sigma_{L'}$ is an $(|S|-1)$-manifold. By Theorem \ref{thm:singerq=1}, $L_\mathbf{1}^2b_{|S|-2}(\Sigma_{L'})=0$, and using the arguments of Lemma \ref{lemma:l2spherecellulations}, we can conclude that $L^2_\mathbf{1} b_{|S|-2}(\Sigma_{L})=0$.
\end{proof}

As a corollary to Theorem \ref{thm:genl2kn}, we also obtain a specialized version of Conjecture \ref{conj:Singer} where $n=4$ and $W$ is $2$-spherical.

\begin{cor}
\label{cor:kns3}
Suppose that $(W,S)$ is $2$-spherical with $|S|\geq 6$ and that the nerve $L$ is a triangulation of $S^3$. Furthermore, suppose that every infinite special subgroup $W_T$, with $|T|=3,4$, is Euclidean or $QL_3$. Then $$L_\Q^2b_k(\Sigma_{L})=0 \text{ for } k<2.$$
\end{cor}
\begin{proof}
Since $L$ is a triangulation of $S^3$, it follows that $\vcd W=4$. In particular, $W$ satisfies the hypothesis of Theorem \ref{thm:genl2kn}.
\end{proof}

\begin{remark}
 Figure \ref{figure:kns3} gives examples of Coxeter diagrams whose corresponding Coxeter system $(W,S)$ has $|S|=6$ and satisfies the hypothesis of Corollary \ref{cor:kns3} (if two vertices are not connected, then the implied label between them is $2$). The author does not know whether there exist examples whenever $|S|\geq 7$.
\end{remark}

\begin{figure}[h]
\[\begin{tikzpicture}
\tikzstyle{vertex}=[circle, draw, inner sep=0pt, minimum size=4pt]
\tikzstyle{every node}=[circle,inner sep=0pt,minimum size=0pt]

\vertex (a) at (0,0) {};
\vertex (b) at (0,1) {};
\vertex (c) at (1,.5) {};
\vertex (d) at (2,.5) {};
\vertex (e) at (3,0) {};
\vertex (f) at (3,1) {};

\path
    (a) edge node[pos=.5,left=0.5pt] (q) {$q$} (b)
    (a) edge node[pos=.5,below=0.5pt] (r) {$r$} (c)
    (b) edge node[pos=.5,above=0.5pt] (s) {$s$} (c)
    (c) edge node[pos=.5,above=0.5pt] (m) {$m$} (d)
    (d) edge node[pos=.5,below=0.5pt] (t) {$t$} (e)
    (d) edge node[pos=.5,above=0.5pt] (u) {$u$} (f)
    (e) edge node[pos=.5,right=0.5pt] (v) {$v$} (f)
;
\end{tikzpicture}\]
$\frac{1}{q} +\frac{1}{r}+\frac{1}{s}=1$, \hskip2mm $\frac{1}{t} +\frac{1}{u}+\frac{1}{v}=1$, \hskip2mm {\small $m=2,3,4$
\vskip2mm
\begin{center}\begin{itemize}
\item If $m=3$, then $s,r,u,t\neq 6$ and either $s,r\neq 4$ or $u,t\neq 4$.
\item If $m=4$, then $s,r,u,t\neq 4,6$.
\end{itemize}\end{center}}
\caption{:\text{ }$2$-spherical Coxeter diagrams satisfying the hypothesis of Corollary \ref{cor:kns3}}
\label{figure:kns3}
\end{figure}

\end{document}